\documentclass[reqno]{amsart}
\usepackage{amscd,amssymb}
\usepackage[all]{xy}
\renewcommand{\mod}{\operatorname{mod}\nolimits}
\DeclareMathOperator{\umod}{\underline{mod}}
\DeclareMathOperator{\omod}{\overline{mod}}

\DeclareMathOperator{\add}{add}

\DeclareMathOperator{\rad}{rad}

\DeclareMathOperator{\Hom}{Hom}
\DeclareMathOperator{\End}{End}
\newcommand{\uHom}{\operatorname{\underline{Hom}}\nolimits}
\newcommand{\oHom}{\operatorname{\overline{Hom}}\nolimits}
\renewcommand{\Im}{\operatorname{Im}\nolimits}

\newcommand{\Coker}{\operatorname{Coker}\nolimits}

\newcommand{\Soc}{\operatorname{Soc}\nolimits}
\newcommand{\Top}{\operatorname{Top}\nolimits}
\newcommand{\Tr}{\operatorname{Tr}\nolimits}
\newcommand{\Ext}{\operatorname{Ext}\nolimits}

\newcommand{\op}{{\operatorname{op}\nolimits}}
\newcommand{\Ab}{{\operatorname{Ab}\nolimits}}
\DeclareMathOperator{\CM}{{\operatorname{CM}\nolimits}}
\DeclareMathOperator{\domdim}{domdim}
\DeclareMathOperator{\gldim}{gldim}
\DeclareMathOperator{\resdim}{resdim}
\DeclareMathOperator{\id}{id}
\DeclareMathOperator{\pd}{pd}

\newcommand{\G}{\Gamma}
\renewcommand{\L}{\Lambda}
\newcommand{\U}{{\mathcal U}}
\newcommand{\Z}{{\mathcal Z}}

\newcommand{\C}{{\mathcal C}}
\newcommand{\D}{{\mathcal D}}

\newcommand{\I}{{\mathcal I}}

\newcommand{\X}{{\mathcal X}}
\newcommand{\Y}{{\mathcal Y}}
\renewcommand{\P}{{\mathcal P}}

\newcommand{\extto}{\xrightarrow}

\newtheorem{lem}{Lemma}[section]
\newtheorem{prop}[lem]{Proposition}
\newtheorem{cor}[lem]{Corollary}
\newtheorem{thm}[lem]{Theorem}
\newtheorem*{Theorem4.7}{Theorem 4.7}
\newtheorem*{Theorem4.5}{Theorem 4.5}
\newtheorem*{Theorem5.3}{Theorem 5.3}
\newtheorem*{Theorem5.10}{Theorem 5.10}
\theoremstyle{definition}
\newtheorem{defin}[lem]{Definition}

\newtheorem{remark}[lem]{Remark}

\newtheorem{example}[lem]{Example}

\makeatletter
\def\blfootnote{\xdef\@thefnmark{}\@footnotetext}
\makeatother

\usepackage{amssymb}
\usepackage{amstext}
\usepackage{amsmath}
\usepackage{amscd}
\usepackage{amsthm}
\usepackage{amsfonts}
\usepackage{enumerate}
\usepackage{graphicx}
\usepackage{latexsym}
\usepackage{mathrsfs}
\usepackage{mathtools}
\usepackage[all]{xy}
\xyoption{all}

\usepackage{pstricks}
\usepackage{lscape}
\usepackage{comment}

\numberwithin{equation}{section}

\begin{document}

\title[Auslander--Gorenstein algebras and precluster tilting]{Auslander--Gorenstein algebras\\
and precluster tilting}

\author[Iyama]{Osamu Iyama}
\address{Osamu Iyama\\ 
Graduate School of Mathematics\\
Nagoya University\\
Chikusa-ku, Nagoya, 464-8602\\
Japan}
\email{iyama@math.nagoya-u.ac.jp}
\author[Solberg]{\O yvind Solberg}
\address{\O yvind Solberg\\
Institutt for matematiske fag\\
NTNU\\ 
N--7491 Trondheim\\ 
Norway}
\email{oyvind.solberg@ntnu.no}

\begin{abstract}
  We generalize the notions of $n$-cluster tilting subcategories and
  $\tau$-selfinjective algebras into $n$-precluster tilting
  subcategories and $\tau_n$-selfinjective algebras,
  where we show that a subcategory naturally associated
  to $n$-precluster tilting subcategories has a higher
  Auslander--Reiten theory.  Furthermore, we give a bijection between
  $n$-precluster tilting subcategories and $n$-minimal Auslander--Gorenstein
  algebras, which is a higher dimensional analog of Auslander--Solberg
  correspondence (Auslander--Solberg, 1993) as well as a Gorenstein analog of
  $n$-Auslander correspondence (Iyama, 2007).  The Auslander--Reiten theory
  associated to an $n$-precluster tilting subcategory is used to
  classify the $n$-minimal Auslander--Gorenstein algebras into four disjoint
  classes. Our method is based on relative homological algebra due to
  Auslander--Solberg.
\end{abstract}



\dedicatory{To the memory of Maurice Auslander.}

\maketitle

\tableofcontents

\section{Introduction}\label{section:1}
\blfootnote{Part of the results were presented at the meeting ICRA
  XIV, Tokyo, Japan, on August 11th, 2010, and at the meeting
  ``Miniconference on selfinjective algebras'', Torun, Poland, on
  December 15th 2010.  Osamu Iyama's work on this project was
  partially supported by JSPS Grant-in-Aid for Scientific Research (B)
  16H03923, (C) 23540045 and (S) 15H05738. \O yvind Solberg's work on
  this project was partially supported by FRINAT grant number 23130
  from the Norwegian Research Council.}  Higher Auslander--Reiten
theory was introduced in \cite{I} by looking at $n$-cluster tilting
subcategories instead of the whole module category.  It is known that
any $n$-cluster tilting subcategory has $n$-almost split sequences,
and that finite $n$-cluster tilting subcategories correspond to
$n$-Auslander algebras \cite{I3}.  Moreover Artin algebras $\L$ which
have $n$-cluster tilting modules are called $n$-representation-finite
\cite{DI} and have been studied. (We do not assume $\gldim\L\le n$ in
this paper in contrast with several earlier papers \cite{HI,IO,IO2}.)
See for example \cite{ABa,HIO,HZ,IW,J1,J2,JK,Jo,Miz,OT} for further
results in higher Auslander--Reiten theory.

In this paper we introduce the weaker notions of $n$-precluster tilting
subcategories (Definition \ref{defin:n-precluster}) and $\tau_n$-selfinjective algebras (Definition \ref{taun selfinjective}). These notions
generalize and unify two seemingly different concepts, namely
(finite) $n$-cluster tilting subcategories \cite{I} and
$\tau$-selfinjective algebras (or $D\Tr$-selfinjective algebras)
\cite{ASIV}.  Both of these concepts are generalizations of algebras
of finite representation type and they are linked through the
Wedderburn correspondence introduced in \cite{A}.  The first class
corresponds to \emph{$n$-Auslander algebras} $\G$, that is, Artin
algebras $\G$ satisfying
\[\domdim\G \geq n + 1 \geq \gldim \G.\] 
The latter class corresponds to algebras $\G$ that are a
Gorenstein analog of Auslander algebras which satisfy
\[\domdim\G \geq 2 \geq \id_\G\G.\]
Here $\domdim$, $\gldim$ and $\id$ denote dominant dimension, global
dimension and injective dimension, respectively.
A central notion in higher dimensional Auslander--Reiten theory is
$n$-cluster tilting subcategories, which are the whole module categories
in the classical case $n=1$.
As in the process of going from classical
Auslander--Reiten theory to higher Auslander--Reiten theory, we
demonstrate that the number $2$ is quite symbolic, as we show that
$n$-precluster-tilting subcategories correspond to 
the following class of Artin algebras.

\begin{defin}
We call an Artin algebra $\G$ \emph{an $n$-minimal Auslander--Gorenstein algebra} if it satisfies
\[\domdim\G \geq n + 1 \geq \id_\G\G.\] 
\end{defin}
We will observe in Proposition \ref{basic properties of minimal AG}(b)
that this condition is left-right symmetric.  Recall that an Artin
algebra $\G$ is called \emph{Gorenstein} (or \emph{Iwanaga--Gorenstein} more precisely) if
$\id_\G\G$ and $\id\G_\G$ are finite \cite{EJ}, and 
a Gorenstein algebra is called \emph{Auslander--Gorenstein} if
the minimal injective coresolution
\[0\to\G\to I^0\to I^1\to\cdots\] 
of the $\Gamma$-module $\Gamma$ satisfies $\pd_\G I^i\le i$ holds for every $i\ge0$
\cite{FGR,AR,AR2}.  Thus our $n$-minimal Auslander--Gorenstein algebras
can be regarded as the most basic class among Auslander--Gorenstein algebras
since $\pd_\G I^i$ has the `minimal' value $0$ for $0\le i\le n$.

We have the following diagrams of generalizations of Auslander algebras
and representation-finite algebras:
\begin{eqnarray*}
&\xymatrix@R1.5em{
\text{Auslander algebras} \ar@{-->}[r]\ar@{-->}[d] &\text{Gorenstein
  algebras $\G$ with $\domdim\G\geq 2 \geq \id_\G\G$}\ar@{-->}[d]\\
\text{$n$-Auslander algebras}\ar@{-->}[r] & \text{$n$-minimal Auslander--Gorenstein algebras}
}&\\
&\xymatrix@R1.5em{
\text{representation-finite algebras} \ar@{-->}[r]\ar@{-->}[d] &\text{$\tau$-selfinjective algebras}\ar@{-->}[d]\\
\text{$n$-representation-finite algebras}\ar@{-->}[r] & \text{$\tau_n$-selfinjective algebras}
}&
\end{eqnarray*}
We introduce the concept of $n$-precluster tilting
subcategories to describe these algebras, where such a subcategory
$\C$ is defined to be a functorially finite generating-cogenerating
subcategory stable under the $n$-th Auslander--Reiten translate and
selforthogonal in the interval $[1, n-1]$ (that is, $\Ext^i(\C,\C) =
0$ for $i= 1,2,\ldots, n-1$).  More precisely, we show the following
results.
\begin{Theorem4.5}
Fix $n\geq 1$.  There is a bijection between
  Morita-equivalence classes of $n$-minimal Auslander--Gorenstein
  algebras and equivalence classes of finite $n$-precluster tilting
  subcategories $\C$ of Artin algebras,
  where the correspondences are given in
  \textup{Propositions \ref{prop:endoring}} and \textup{\ref{prop:redtoendoring}}.
\end{Theorem4.5}

For a Gorenstein algebra $\G$, we denote by
\[\CM\G=\{X\in\mod\G\mid\Ext^i_\G(X,\G)=0 \text{\ for\ }i>0\}\]
the category of \emph{maximal Cohen--Macaulay} $\G$-modules.  We denote by
$\underline{\CM}\,\G$ the stable category of $\CM\G$, that is, $\CM\G$
modulo the ideal in $\CM\G$ generated by $\add\G$. It is basic that
$\underline{\CM}\,\G$ forms a triangulated category.
On the other hand, for an $n$-precluster tilting subcategory $\C$ of $\mod\L$, let
\begin{eqnarray*}
\Z(\C)&=&\{ X\in\mod\L\mid \Ext^i_\L(\C,X)=0 \text{\ for\ }
  i=1,2,\ldots,n-1\},\\
\U(\C)&=&\Z(\C)/[\C],
\end{eqnarray*}
where $\Z(\C)/[\C]$ denotes the category $\Z(\C)$ modulo the ideal
$[\C]$ in $\Z(\C)$ generated by the subcategory $\C$. 
Note that $\Z(\C)=\mod\L$ for the case $n=1$. We show that the category
$\Z(\C)$ is a Frobenius category whose stable category is $\U(\C)$ (Proposition
\ref{Z(C) is Frobenius}).
Hence $\U(\C)$ forms a triangulated category. In fact it is an
analog of Calabi--Yau reduction of triangulated categories \cite{IY}.
   
\begin{Theorem4.7}
  Given an Artin algebra $\L$ with a finite $n$-precluster tilting
  subcategory $\C=\add M$, let $\G=\End_\L(M)$ be the
  corresponding $n$-minimal Auslander--Gorenstein algebra
  (see \textup{Theorem \ref{AS correspondence}}).  Then $\Z(\C)$ and
  $\CM\G$ are dual categories via the functors
  $\Hom_\L(-,M)\colon\Z(\C)\to\CM\G$ and
  $\Hom_\G(-,M)\colon\CM\G\to\Z(\C)$.  Moreover these functors induce triangle
  equivalences between $\U(\C)$ and $(\underline{\CM}\,\G)^{\op}$.
\end{Theorem4.7}
Furthermore we show that there is a higher Auslander--Reiten theory
also for $n$-precluster tilting subcategories, though with some
differences.  The first difference is that one cannot define $n$-fold
almost split sequences as for $n$-cluster tilting subcategories, but
one is forced to introduce $n$-fold almost split extensions (see
Definition \ref{def:nfoldalmostsplitextension}).  This is because in
this more general setting there does not exist a unique exact sequence
representing this extension.  Namely, we have the following results,
where we denote by $\P(\L)$ (respectively $\I(\L)$) the category of
finitely generated projective (respectively injective) $\L$-modules.

\begin{Theorem5.10}
Let $\C$ be an $n$-precluster tilting subcategory of $\mod\L$,
$X$ an indecomposable module in $\Z(\C)\setminus \P(\L)$, and $Y:=\tau_n(X)$ the corresponding
indecomposable module in $\Z(\C)\setminus \I(\L)$.
\begin{enumerate}[\rm(a)]
\item For each $0\le i\le n-1$,
an $n$-fold almost split extension in $\Ext^n_\L(X,Y)$ can be represented as
\[0\to Y\to C_{n-1}\to \cdots \to C_{i+1}\to Z_i\to C_{i-1}\to \cdots \to C_0\to X\to 0\]
with $Z_i$ in $\Z(\C)$ and $C_j$ in $\C$ for each $j\neq i$.
\item The following sequences are exact.\begin{align*}
&0\to \Hom_\L(\C,Y)\to \Hom_\L(\C,C_{n-1})\to\cdots\to \Hom_\L(\C,C_{i+1})\to\Hom_\L(\C,Z_i)\\
&\qquad\qquad\qquad\to\Hom_\L(\C,C_{i-1})\to\cdots\to\Hom_\L(\C,C_0)\to\rad_\L(\C,X)\to 0,\\
&0\to \Hom_\L(X,\C)\to \Hom_\L(C_0,\C)\to\cdots\to\Hom_\L(C_{i-1},\C)\to\Hom_\L(Z_i,\C)\\
&\qquad\qquad\qquad\to\Hom_\L(C_{i+1},\C)\to\cdots\to\Hom_\L(C_{n-1},\C)\to\rad_\L(Y,\C)\to 0.
\end{align*}
\item If $X$ and $Y$ do not belong to $\C$, then the $n$-fold almost split extension in \emph{(a)}
can be given as a Yoneda product of a minimal projective resolution of $X$ in $\Z(\C)$
\begin{equation*}
0\to\Omega_{\Z(\C)}^{i}(X)\to C_{i-1}\to \cdots \to C_{0}\to X\to 0,
\end{equation*}
an almost split sequence in $\Z(\C)$
\begin{equation*}
0\to\Omega_{\Z(\C)}^{-(n-i-1)}(Y)\to Z_i\to\Omega_{\Z(\C)}^{i}(X)\to 0,
\end{equation*}
and a minimal injective coresolution of $Y$ in $\Z(\C)$
\begin{equation*}
0\to Y\to C_{n-1}\to \cdots \to C_{i+1}\to\Omega_{\Z(\C)}^{-(n-i-1)}(Y)\to 0.
\end{equation*}
\end{enumerate}
\end{Theorem5.10}

We use this Auslander--Reiten theory to classify the
$n$-minimal Auslander--Goren\-stein algebras into four disjoint classes,
as was done for the case $n=1$ in \cite{ASIV}.

We refer to related results.  In \cite{CK}, the authors studied
$(n,m)$-ortho-symmetric modules, where our $n$-precluster tilting
modules are precisely $(n-1,0)$-ortho-symmetric.  In \cite{K}, the
author proved Theorem \ref{AS correspondence} independently.

The paper is organized as follows.  In the second section we recall
the relative homological algebra over Artin algebras we need (see
\cite{ASI, ASII, ASIII}) and some unpublished results of Maurice
Auslander and the second author.  In Section $3$ the notions of
$n$-precluster tilting subcategories and $\tau_n$-selfinjective
algebras are introduced and their basic properties are discussed.  We
show in the next section that there is a one-to-one correspondence
between finite $n$-precluster subcategories and $n$-minimal
Auslander--Gorenstein Artin algebras, where the $n$-Auslander algebras
are characterized within this class.  In the fifth section we show
that there is a meaningful higher Auslander--Reiten theory in
$n$-precluster tilting subcategories also.  This theory is transferred
in the next section over to the subcategory of maximal Cohen--Macaulay
modules over the $n$-minimal Auslander--Gorenstein Artin algebras.  In
the final section we use higher Auslander--Reiten theory to classify
the $n$-minimal Auslander--Gorenstein Artin algebras into four
disjoint classes.

\medskip\noindent{\bf Notations. } Throughout the paper, all modules
are left modules.  The composition of morphisms $f\colon X\to Y$ and
$g\colon Y\to Z$ is denoted by $gf\colon X\to Z$.

Let $R$ be a commutative Artinian ring and $\L$ an Artin $R$-algebra.
We denote by $\mod\L$ the category of finitely generated $\L$-modules,
by $\P(\L)$ (respectively $\I(\L)$) the category of finitely generated
projective (respectively injective) $\L$-modules, and by
$D\colon\mod\L\leftrightarrow\mod\L^{\op}$ the duality $\Hom_R(-,E)$,
where $E$ is the injective hull of the $R$-module $R/\rad R$.  We
denote by $\tau\colon\umod\L\to\omod\L$ the Auslander--Reiten
translation.

Let $\X$ be a full subcategory of $\mod\L$. We call
$\X$ a \emph{generator} (respectively
\emph{cogenerator}) if $\L\in\X$ (respectively $D\L\in\X$).  A
\emph{right $\X$-approximation} of $A\in\mod\L$ is a morphism
$f\colon X\to A$ with $X\in\X$ such that any morphism
$g\colon Y\to A$ with $Y\in\X$ factors through $f$.  We call
$\X$ \emph{contravariantly finite} if any $A\in\mod\L$ has a right
$\X$-approximation.  Dually, we define a \emph{left
  $\X$-approximation} and a \emph{covariantly finite} subcategory.  A
contravariantly and covariantly finite subcategory is called
\emph{functorially finite}.

\section{Preliminaries on relative homological algebra}\label{section:2} 

A systematic study of relative homological algebra over Artin algebras
was carried out in \cite{ASI, ASII, ASIII}. We recall the relevant
background and results, and in addition we
give some unpublished results of Maurice Auslander and the second
author.

\subsection{Relative homological algebra} We start with the setup for
relative homological algebra, where we assume throughout that $\L$ is
an Artin algebra. Relative homological algebra for us begins with
defining a set of exact sequences, and this is done through giving an
additive sub-bifunctor of $\Ext^1_\L(-,-)$ (see \cite{ASI} for further
details).  Let
\[F\subseteq \Ext^1_\L(-,-)\colon (\mod\L)^\op\times\mod\L \to \Ab\]
be an additive sub-bifunctor. Such an additive sub-bifunctor is
nothing else than, for each pair of $\L$-modules $C$ and $A$, a chosen
set of short exact sequences, $F(C,A)$, starting in $A$ and ending in
$C$, which is closed under pullbacks, pushouts and Baer sums (or
direct sums of short exact sequences). 

\begin{defin}
  An exact sequence $\eta\colon 0\to A\to B\to C\to 0$ is said to be
  \emph{$F$-exact} if $\eta$ is in $F(C,A)$.
\end{defin}

In the rest, we fix a subcategory $\X$ of $\mod\L$.  We consider the following
collection $F_\X(C,A)$ of short exact sequences given for a pair of
modules $A$ and $C$ in $\mod\L$:
\begin{multline}
F_\X(C,A) = \{ 0\to A\to B\to C\to 0\in\Ext^1_\L(C,A)\mid\notag\\
 \Hom_\L(X,B)\to \Hom_\L(X,C)\to 0 \textrm{\ exact for all $X$ in
   $\X$}\}.
\end{multline}
Dually one defines $F^\X$. By \cite[Proposition 1.7]{ASI} these collections induce
additive sub-bifunctors $F_\X$ and $F^\X$ of $\Ext^1_\L(-,-)$, and
\begin{equation}\label{defect formula}
F_\X=F^{\tau \X}\ \mbox{ and }\ F_{\tau^-\X}=F^\X
\end{equation}
by \cite[Proposition 1.8]{ASI}. 
If $\X=\add X$ holds for some $X$ in $\mod\L$, we denote $F_{\X}$
and $F^{\X}$ by $F_X$ and $F^X$ respectively.

We can endow $\mod\L$ with a new exact structure by the following result,
where we call a full subcategory $\C$ of $\mod\L$ \emph{$F$-extension closed} if,
for every $F$-exact sequence $0\to A\to B\to C\to0$ such that $A$ and $C$ are in $\C$,
$B$ is in $\C$.

\begin{prop}[\cite{DRSS}]\label{from F to exact}
Let $\X$ be a full subcategory of $\mod\L$ and $F=F^\X$ (respectively $F=F_\X$).
Then $\mod\L$ has a structure of an exact category whose short exact
sequences are precisely the $F$-exact sequences.
More generally, any $F$-extension closed subcategory $\C$ of $\mod\L$
has a structure of an exact category whose short exact sequences are precisely
the $F$-exact sequences contained in $\C$.
\end{prop}

\begin{proof}
  The first assertion follows from \cite[Propositions 1.4 and 1.7]{DRSS}.  The
  second assertion is a general property of extension closed
  subcategories of an exact category.
\end{proof}

We denote by $(\mod\L,F)$ and $(\C,F)$ the exact categories given in
Proposition \ref{from F to exact}.

Recall the following definitions from \cite{ASI}, which coincide with
  the corresponding notions in the exact category $(\mod\L,F)$.

\begin{defin}
\begin{enumerate}[(i)]
\item
A $\L$-module $P$ is said to be \emph{$F$-projective} if
  all $F$-exact sequences $0\to A\to B\to P\to 0$ split. The full
  subcategory of $\mod\L$ consisting of all $F$-projective modules is
  denoted by $\P(F)$.
\item
A $\L$-module $I$ is said to be \emph{$F$-injective} if
  all $F$-exact sequences $0\to I\to B\to C\to 0$ split. The full
  subcategory of $\mod\L$ consisting of all $F$-injective modules is
  denoted by $\I(F)$.
\item
$F$ has \emph{enough $F$-projectives}
  (respectively \emph{$F$-injectives}) if for each $C$ (respectively $A$) in $\mod\L$ there exists an
  $F$-exact sequence 
  \[0\to C'\to P\to C\to 0\] 
  with $P$ in $\P(F)$ (respectively $0\to A\to I\to A'\to 0$ with $I$ in
  $\I(F)$).
\end{enumerate}
\end{defin}

The $F_\X$-projectives and $F_\X$-injectives are given as
\begin{equation}\label{P(FX) and I(FX)}
\P(F_\X)=\add\{\X, \P(\L)\}\ \mbox{ and }\ \I(F_\X)=\add\{\tau\X, \I(\L)\}
\end{equation}
by \cite[Proposition 1.10]{ASI}. 
For example, we have $\P(F_{\L})=\P(\Ext^1_\L(-,-))=\P(\L)$ and $\I(F^{D\L})=\I(\Ext^1_\L(-,-))=\I(\L)$.
Furthermore an additive sub-bifunctor
$F$ of $\Ext^1_\L(-,-)$ has enough $F$-projectives and $F$-injectives
if and only if $\P(F)$ is functorially finite in $\mod\L$ and
$F=F_{\P(F)}$ (see \cite[Corollary 1.13]{ASI}).  In this case we
denote by $\Omega_F^1(X)$ the kernel of the $F$-projective cover of $X$.
The $\L$-module $\Omega_F^{-1}(Y)$ is defined dually.

Assume from now on that $F$ is an additive sub-bifunctor of
$\Ext^1_\L(-,-)$ with enough $F$-projectives and $F$-injectives.  Recall that
an exact sequence
\[\cdots\to C_{i+1}\xrightarrow{f_{i+1}}C_i\xrightarrow{f_i}C_{i-1}\to\cdots\]
is called \emph{$F$-exact} if all the short exact sequences $0\to\Im
f_{i+1}\to C_i\to\Im f_i\to 0$
are $F$-exact.  Given two modules $A$ and $C$ in
$\mod\L$, there exist $F$-exact sequences
\[\mathbb{P}\colon \cdots\to P_2\to P_1\to P_0\to C\to 0\]
with $P_i$ in $\P(F)$ and 
\[\mathbb{I}\colon 0\to A\to I^0\to I^1\to I^2\to \cdots\]
with $I^j$ in $\I(F)$. We call these exact sequences an
\emph{$F$-projective resolution} and an \emph{$F$-injective coresolution} of $C$
and $A$, respectively. 
For $i\ge1$, the $i$-th homologies of the complexes $\Hom_\L(\mathbb{P},A)$ and
$\Hom_\L(C,\mathbb{I})$ are isomorphic, and denoted by $\Ext^i_F(C,A)$
and called the $i$-th $F$-relative extension group
(see \cite[Section 2]{ASI}). Then $\Ext^1_F(-,-)$ is naturally identified with $F$. Using $\Ext^i_F(-,-)$,
we define $F$-relative projective dimension,
$F$-relative injective dimension and $F$-relative global dimension as in the
absolute setting and the basic properties are the same in the $F$-relative
setting. On the other hand, we denote by
\[\uHom_F(A,C)\ \mbox{ (respectively $\oHom_F(A,C)$)}\]
$\Hom_\L(A,C)$ modulo all the homomorphisms factoring
through an $F$-projective (respectively $F$-injective) module.
The \emph{$F$-stable category}
\[\umod_F\L\ \mbox{ (respectively $\omod_F\L)$}\]
has the same objects as $\mod\L$, and the morphism
  sets are given by $\uHom_F(A,C)$ (respectively $\oHom_F(A,C)$).
For $F=F_{\X}$, the Auslander--Reiten translation
$\tau\colon\umod\L\simeq\omod\L$ induces the $F$-relative Auslander--Reiten translation\begin{equation}\label{relative AR}
\tau\colon\umod_F\L\simeq\omod_F\L.
\end{equation}
  Another central result that has an analog in the $F$-relative setting
  is the Auslander--Reiten formula, which we recall next.

\begin{prop}[\protect{\cite[Proposition 2.3]{ASI}}]\label{prop:AR-formula}
  Let $F$ be an additive sub-bifunctor of $\Ext_\L^1(-,-)$ with enough
  $F$-projectives (and $F$-injectives).  Then for all modules $A$ and $C$ in
  $\mod\L$ we have an isomorphism
\[\Ext^1_F(C,\tau A )\simeq D\uHom_{F}(A,C).\]
\end{prop}

In general the higher $F$-relative extension groups $\Ext^i_F(C,A)$ are
not necessarily related to the higher absolute extension groups
$\Ext^i_\L(C,A)$. However, in some situations one can compute the
absolute extensions by $F$-relative ones, as described in the next result.

\begin{prop}[\protect{\cite[Proposition 1.3]{L}}]\label{prop:relative=absolute}
Let $\X$ be a functorially finite subcategory of $\mod\L$.  
\begin{enumerate}[\rm(a)]
\item
A module $C$ in $\mod\L$
satisfies $\Ext^i_\L(C,\X)=0$ for $0<i<n$ if and only if
$\Ext^i_{F^\X}(C,A)=\Ext^i_\L(C,A)$ holds for $0<i<n$ and for all $A$ in $\mod\L$. 
\item
A module $A$ in $\mod\L$
satisfies $\Ext^i_\L(\X,A)=0$ for $0<i<n$ if and only if
$\Ext^i_{F_\X}(C,A)=\Ext^i_\L(C,A)$ for $0<i<n$ and for all $C$ in $\mod\L$.  
\end{enumerate}
\end{prop} 
\begin{proof}
  We only show (a). This is proved in \cite[Proposition 1.3]{L} under
  the assumption that $\X$ is a cogenerator.
  We can drop it by considering $\X'=\add\{\X,\I(\L)\}$ and using
  $F^\X=F^{\X'}$.
\end{proof}

\subsection{The Auslander--Reiten translation revisited} 
Next we recall some unpublished results of Maurice Auslander and the
second author that we need later.  We show that the Auslander--Reiten
translation $\tau$ gives a bijection between $\Ext^1_F(C,A)$ and
$\Ext^1_G(\tau(C),\tau(A))$ for certain $F$ and $G$, and moreover
$\tau$ commutes with relative syzygies.

The transpose is a duality $\Tr\colon \umod\L\to \umod\L^\op$ \cite{ABr}.
However, from an exact sequence 
\[0\to A\to B\to C\to 0\] 
in $\mod\L$, there is not necessarily a naturally associated exact
sequence
\[0\to \Tr C\to \Tr B\oplus X\to \Tr A\to 0\] 
in $\mod\L^{\op}$. We show that when restricting to appropriate classes of exact
sequences we have such a natural correspondence.

\begin{prop}\label{lem1.1}
  Let $\X$ be a functorially finite generator-cogenerator in $\mod\L$,
  and let $F=F^\X$, $G^\op=F^{D(\X)}$ and $G=F_\X$.
\begin{enumerate}[\rm(a)]
\item If $0\to A\to B\to C\to 0$ is $F$-exact in $\mod\L$, then there
  are a $G^\op$-exact sequence
\[0\to \Tr C\to \Tr B\oplus P\to \Tr A\to 0\] in $\mod\L^\op$ for some
projective $\L^\op$-module $P$ and a $G$-exact sequence
\[0\to \tau(A)\to \tau(B)\oplus I\to \tau(C)\to 0\]
in $\mod\L$ for some injective $\L$-module $I$.
\item For all $A$ and $C$ in $\mod\L$, we have functorial isomorphisms
\[F(C,A)\simeq G^{\op}(\Tr A,\Tr C)\simeq G(\tau(C),\tau(A)).\]
\end{enumerate}
\end{prop}
\begin{proof} 
(a) Let $\eta\colon 0\to A\to B\extto{f} C\to 0$ be an $F$-exact
sequence. By the Horseshoe Lemma we have the following commutative
diagram
\[\xymatrix@R=1em{
0\ar[r] & P_1\ar[r]\ar[d] & P_1\oplus Q_1\ar[r]\ar[d] & Q_1\ar[r]\ar[d] & 0\\
0\ar[r] & P_0\ar[r]\ar[d] & P_0\oplus Q_0\ar[r]\ar[d] & Q_0\ar[r]\ar[d] & 0\\
0\ar[r] & A  \ar[r]\ar[d] & B\ar[r]\ar[d]   & C\ar[r]\ar[d]   & 0\\
        & 0 & 0 & 0 &
}\]
with $P_i$ and $Q_i$ projective $\L$-modules for $i=0,1$.
Let $(-)^*=\Hom_\L(-,\L)$.
Then this induces the following
commutative diagram
\[\xymatrix@R=1em{
        & 0\ar[d]           & 0\ar[d]           & 0\ar[d]          & \\
0\ar[r] & C^*\ar[r]\ar[d]    & B^*\ar[r]\ar[d]   & A^*\ar[r]\ar[d]  & 0\\
0\ar[r] & Q^*_0\ar[r]\ar[d] & P^*_0\oplus Q^*_0\ar[r]\ar[d] & P^*_0\ar[r]\ar[d] & 0\\
0\ar[r] & Q^*_1\ar[r]\ar[d] & P^*_1\oplus Q^*_1\ar[r]\ar[d] & P^*_1\ar[r]\ar[d] & 0\\
\eta'\colon 0\ar[r] & \Tr C\ar[r]\ar[d] & \Tr B\oplus P\ar[r]^{g}\ar[d] 
                                                & \Tr A\ar[r]\ar[d] & 0\\
        & 0 & 0 & 0 &
}\]
where $P$ is a projective $\L^\op$-module. The columns are exact by the
definition of $\Tr$, and the top row is exact because $\eta$ is $F$-exact
and $\X$ contains $\L$. Since the second and the third rows are split exact,
the Snake Lemma shows that the bottom row $\eta'$ is also exact.
It remains to show that $\eta'$ is $G^{\op}$-exact.
Then the dual of this sequence is $G$-exact.

Since $\eta$ is $F$-exact, the map
$f\colon\uHom_{\L}(B,\X)\to\uHom_{\L}(A,\X)$ is surjective.  Since
$\Tr$ is a duality, the map
$g=\Tr f\colon\uHom_{\L^{\op}}(\Tr\X,\Tr
B)\to\uHom_{\L^{\op}}(\Tr\X,\Tr A)$ is surjective.  By a standard
argument, the map
$g\colon\Hom_{\L^{\op}}(\Tr\X,\Tr B\oplus
P)\to\Hom_{\L^{\op}}(\Tr\X,\Tr A)$ is also surjective.  Since
$\tau\Tr\X=D(\X)$ and hence $F_{\Tr\X}=G^{\op}$ holds, $\eta'$ is
$G^{\op}$-exact.

(b) By (a), we have a morphism $F(C,A)\to G^{\op}(\Tr A,\Tr C)$ of bifunctors.
Since the same argument gives the inverse morphism $G^{\op}(\Tr A,\Tr C)\to F(C,A)$ of bifunctors,
we have the first desired isomorphism. The second isomorphism follows
immediately by applying the duality $D$.
\end{proof}

We end this subsection by showing how this induces isomorphisms on
relative extension groups and relative stable homomorphism sets.
\begin{thm}\label{prop:exttauisom}\footnote{Theorem \ref{prop:exttauisom} was obtained by Maurice
    Auslander and the second author, but they never got to be
    published in printed form before now. However, the results were
    presented at a seminar at the University of Bielefeld, Germany.}    
Let $\X\subseteq \mod\L$ be a functorially finite
generator-cogenerator, and let $F=F^\X$ and $G=F_\X$ be the
corresponding additive sub-bifunctors of $\Ext^1_\L(-,-)$. Then the following
is true.
\begin{enumerate}[\rm(a)]
\item $\tau$ gives an equivalence $\umod_F\L\simeq\umod_G\L$.
\item The following diagrams commute up to isomorphisms of functors.
\[\xymatrix@R=1.5em{\umod_F\L\ar[r]^{\tau}\ar[d]^{\Omega_F}&\umod_G\L\ar[d]^{\Omega_G}\\
\umod_F\L\ar[r]^{\tau}&\umod_G\L}\ \ \ \ \ 
\xymatrix@R=1.5em{\omod_F\L\ar@{<-}[r]^{\tau^-}\ar@{<-}[d]^{\Omega^-_F}&\omod_G\L\ar@{<-}[d]^{\Omega^-_G}\\
\omod_F\L\ar@{<-}[r]^{\tau^-}&\omod_G\L}\]
\item $\tau$ induces a functorial isomorphism in both variables 
\[\varphi_n=\varphi_{C,A,n}\colon \Ext^n_F(C,A)\simeq
\Ext^n_G(\tau(C),\tau(A))\]
for all pairs of $A$ and $C$ in $\mod\L$ and $n\geq 1$. 
\end{enumerate}
\end{thm}
\begin{proof} (a) We have an equivalence
$\tau\colon\umod_F\L\simeq\omod_F\L$ in \eqref{relative AR}.
Since $\X$ is a generator-cogenerator, $\I(F)=\P(G)$ holds by \eqref{P(FX) and I(FX)}.
Thus $\omod_F\L=\umod_G\L$ holds, and the assertion follows.

(b) We only prove commutativity of the left diagram. Let
\[0\to \Omega_F(A)\to P\to A\to 0\] 
be $F$-exact with $P$ in $\P(F)$. By Proposition~\ref{lem1.1}, we have
a $G$-exact sequence
\[0\to \tau\Omega_F(A)\to \tau(P)\oplus I\to\tau(A)\to 0\]
with $I$ in $\I(\L)$. Since $\tau(P)\oplus I$ is in $\P(G)$ by \eqref{P(FX) and I(FX)}, we have
an isomorphism $\tau\Omega_F\simeq\Omega_G\tau$ of functors.

(c) We have the following functorial isomorphisms
\begin{align}\notag
\Ext^n_F(C,A) &\simeq \Ext^1_F(\Omega_F^{n-1}(C),A) \\ \notag
              &\simeq \Ext^1_G(\tau \Omega_F^{n-1}(C),\tau(A))\\ \notag
	      &\simeq \Ext^1_G(\Omega_G^{n-1}\tau(C),\tau(A))\\ \notag
              & \simeq \Ext^n_G(\tau(C),\tau(A))\notag
\end{align}
by dimension shift, Proposition \ref{lem1.1}(b) and (b), where all the involved isomorphisms are functorial. The claim follows. 
\end{proof}

\subsection{Relative tilting theory} Tilting theory is an important
topic in representation theory of Artin algebras and elsewhere. It
also has a relative version, which we recall from \cite{ASII,ASIII}.  Here
we always assume that our additive sub-bifunctor $F$ has enough $F$-projectives and
enough $F$-injectives.  
For a subcategory $\C$ in $\mod\L$, let
\begin{eqnarray*}
\C^{\perp_F}&=&\{M\in\mod\L\mid \Ext^i_F(\C,M)=0 \text{\    for all $i>0$}\},\\
{}^{\perp_F}\C &=&\{M\in\mod\L\mid \Ext^i_F(M,\C)=0 \text{\    for all $i>0$}\}.
\end{eqnarray*}
When $F=\Ext^1_\L(-,-)$, we simply denote $\C^{\perp_F}$ and ${}^{\perp_F}\C$ by $\C^{\perp}$ and ${}^{\perp}\C$ respectively.

\begin{defin}
We call $T$ in $\mod\L$ an \emph{$F$-cotilting module} if 
\begin{enumerate}[(i)]
\item[(i)] $\id_F T<\infty$,
\item[(ii)] $\Ext^i_F(T,T)=0$ for $i>0$, 
\item[(iii)] for all $I$ in $\I(F)$ there exists an $F$-exact sequence 
\[0\to T_n\to T_{n-1}\to \cdots \to T_1\to T_0\to I\to 0\]
with $T_i$ in $\add T$. 
\end{enumerate}
\end{defin}
It is shown in \cite[Corollary 3.14]{ASII} that if there exists an
$F$-cotilting module $T$, then $\P(F)$ and hence $\I(F)$ are of finite
type.  In fact, they contain the same number of non-isomorphic
indecomposable objects as $\add T$. 

Next we collect the basic results on relative cotilting modules that
we need later. 
\begin{thm}[\protect{\cite{ASII}}]\label{thm:relativecotilting}
  Let $X$ be a generator in $\mod\L$ and $F=F_{X}$.
  Let $T$ be an $F$-cotilting $\L$-module and $\G=\End_\L(T)$.
  Then we have the following.
\begin{enumerate}[\rm(a)]
\item $\L\simeq \End_\G(T)$. 
\item Any $C$ in $^{\perp_F} T$ has an $F$-exact sequence
\[0\to C\to T_0\extto{f_0} T_1 \extto{f_1} T_2\extto{f_2} T_3 \to
\cdots\]
with $T_i$ in $\add T$ and $\Im f_i$ in $^{\perp_F}T$ for all $i\geq 0$.
\item $\Ext_F^i(C,A)\simeq
    \Ext_\G^i(\Hom_\L(A,T),\Hom_\L(C,T))$ for all modules $A$ and $C$
    in ${^{\perp_F}T}$ and $i\geq 0$. 
\item The module $U=\Hom_\L(X,T)$ is a
    cotilting $\G$-module with $\id_F T\leq \id_\G U\leq \id_F T + 2$
    \footnote{This inequality is a corrected version of \cite[Theorem 3.13(d)]{ASII} due to an error in the statement of \cite[Proposition 3.11]{ASII}, where the bound should be $\id_F T + 2$, instead of $\max\{\id_FT,2\}$.}.
\item $\Hom_\L(-,T)\colon \mod\L\to \mod\G$ and
    $\Hom_\G(-,T)\colon \mod\G\to \mod\L$ induce quasi-inverse dualities
    $\Hom_\L(-,T)\colon{^{\perp_F}T}\to{^\perp U}$ and
    $\Hom_\G(-,T)\colon{^\perp U}\to{^{\perp_F}T}$.
\end{enumerate}
\end{thm}
\begin{proof}
(a) is \cite[Corollary 3.4]{ASII}, (b) is \cite[Theorem 3.2(a)]{ASII},
(c) is \cite[Proposition 3.7]{ASII}, (d) is \cite[Theorem 3.13(d)]{ASII},
and (e) is \cite[Corollary 3.6(a), Proposition 3.8(b)]{ASII}.
\end{proof}

Now let $F$ be an additive sub-bifunctor of $\Ext^1_\L(-,-)$, and $\X$
a generator of $\mod\L$.  Then
$\X$-$\resdim_F M$ is defined to be the infimum of $n$ such that there
exists an $F$-exact sequence
\[0\to X_n \to X_{n-1}\to\cdots\to X_1\to X_0\to M\to 0,\] 
where $X_i$ is in $\X$ for all $i\geq 0$. We denote by
$\widehat{\X}$ the full subcategory of $\mod\L$ consisting of all $M$
in $\mod\L$ with $\X$-$\resdim_F M<\infty$.  Let
$\X$-$\resdim_F(\mod\L):=\sup\{\X$-$\resdim_F M\mid M\in \mod\L\}$.

The following result is taken from \cite{ASIII} (with the exception of
(c)), which connects special direct summands of absolute cotilting
modules with relative cotilting modules.
Let $\G$ be an Artin algebra and $X$ in $\mod\G$. Recall that
a direct summand $Y$ of $X$ is said to be a \emph{dualizing summand of $X$}
if there exists an exact sequence
\[0\to X\extto{f} Y^0\to Y^1\] 
with $f$ a left $(\add Y)$-approximation and $Y^i$ in $\add Y$.  This
is shown to be equivalent to the natural homomorphism
\[X\to \Hom_\L(\Hom_\G(X,Y),Y)\] 
being an isomorphism, where $\L=\End_\G(Y)$ \cite[Proposition 2.1]{ASIII}.

\begin{thm}\label{thm:dualizingsummand}
  Let $\G$ be an Artin algebra and $U$ a cotilting $\G$-module.
  For a dualizing summand $T$ of $U$, let $\L=\End_\G(T)$, $X=\Hom_\G(U,T)$
  and $F=F_X$ an additive sub-bifunctor of $\Ext^1_\L(-,-)$.
  Then the following assertions hold.
\begin{enumerate}[\rm(a)]
\item $\G\simeq \End_\L(T)$. 
\item $T$ is an $F$-cotilting $\L$-module with $\id_FT\leq
  \max\{\id {_\G U}, 2\}$.
\item If $T$ is injective as a $\G$-module, then $\id_FT\leq \max\{\id {_\G U}-2,0\}$. 
\end{enumerate}
\end{thm}
\begin{proof}
(a) is \cite[Proposition 2.4(b)]{ASIII}, and (b) is \cite[Proposition 2.7(c)]{ASIII}.

  (c) Let $r=\id {_\G U}$ and $t=\max\{ r- 2, 0\}$. We prove that
  $({^{\perp_F}T})$-$\resdim_F(\mod\L) \leq t$, as this implies that $\id_F T \leq t$.
  This is done by showing that $\Omega^t_F(C)$ is in ${^{\perp_F}}T$ for all $C$ in $\mod\L$.

Assume that $_\G T$ is injective.  Then by \cite[Lemma 2.4]{ASIV}
the module $_\L T$ is a cogenerator in $\mod\L$, and therefore 
\[C\simeq \Hom_\G(\Hom_\L(C,T),T)\] 
for all modules $C$ in $\mod\L$.  In particular, $C\simeq
\Hom_\G(B,T)$ for some $\G$-module $B$. Let
\[\mathbb{P}\colon B\extto{f^0} U^0\extto{f^1} U^1\extto{f^2} U^2\extto{f^3} \cdots,\]
be a sequence of minimal left $(\add U)$-approximations of $B$, and let
$B^0=B$ and $B^j=\Coker f^{j-1}$ for $j\ge1$.
Then an $F$-projective resolution of $C$ is given by
\[\Hom_\G(\mathbb{P},T)\colon\cdots\to \Hom_\G(U^2,T)\to \Hom_\G(U^1,T)\to\Hom_\G(U^0,T)\to C\to 0,\]
and therefore we have, for every $j\ge0$,
\[\Omega^j_F(C)=\Hom_\G(B^j,T).\]

Let $\Sigma=\End_\G(U)$.  Then $U$ is a cotilting $\Sigma$-module with
$\id {_\Sigma U} = \id {_\G U}$.  The above complex $\mathbb{P}$ gives rise
to a projective resolution of the $\Sigma$-module $\Hom_\G(B,U)$,
\[\cdots\to \Hom_\G(U^2,U)\to \Hom_\G(U^1,U)\to\Hom_\G(U^0,U)\to \Hom_\G(B,U)\to 0.\]
Given a projective presentation $F_1\to F_0\to B\to 0$ of the $\G$-module $B$,
it induces an exact sequence 
\[0\to \Hom_\G(B,U)\to \Hom_\G(F_0,U)\to \Hom_\G(F_1,U)\to B'\to 0\]
of $\Sigma$-modules with $\Hom_\G(F_0,U)$ and $\Hom_\G(F_1,U)$ in
$\add {_\Sigma U}$. Then $\Omega^{j+2}_\Sigma(B')=\Hom_\G(B^j,U)$
holds for every $j\ge0$.

Now we show $\Hom_\G(B^t,U)$ is in ${^\perp_\Sigma U}$.
If $r \leq 2$, then $\Hom_\G(B^t,U)=\Hom_\G(B,U)=\Omega^2_\Sigma(B')$
is in ${^\perp_\Sigma U}$ since $\id {_\G U}=r\le 2$.
If $r > 2$, then $\Hom_\G(B^t,U)=\Hom_\G(B^{r-2},U)=\Omega^r_\Sigma(B')$
is in ${^\perp_\Sigma U}$ since $\id {_\G U}=r$.

By (a), (b) and Theorem \ref{thm:relativecotilting}(e), we have dualities
\[{^\perp_\Sigma U}\xleftarrow{\Hom_\G(-,U)}{^\perp_\G U}
\xrightarrow{\Hom_\G(-,T)}{^{\perp_F} T}.\]
We take $B''$ in ${^\perp_\G U}$ such that $\Hom_\G(B'',U)\simeq \Hom_\G(B^t,U)$
as $\Sigma$-modules. Then
\[\Omega^t_F(C)=\Hom_\G(B^t,T)\simeq\Hom_\G(B'',T)\in{^{\perp_F} T}\]
as $\L$-modules. Thus the claim holds.
\end{proof}

Let $F$ be an additive sub-bifunctor of $\Ext^1_\L(-,-)$.  A full
subcategory $\X$ of $\mod\L$ is \emph{$F$-resolving} (respectively
\emph{$F$-coresolving}) if
\begin{enumerate}[\rm(i)]
\item $\X$ is $F$-extension closed,
\item $\P(F)$ (respectively $\I(F)$) is contained in $\X$,
\item if $0\to A\to B\to C\to 0$ is $F$-exact and $B$ and $C$
  are in $\X$ (respectively $A$ and $B$ are in $\X$), 
  then $A$ (respectively $C$) is in $\X$.
\end{enumerate}

We need the following preparation from Auslander--Buchweitz theory.

\begin{prop}[\protect{\cite[Theorems 2.4, 2.5, Proposition 2.2]{ASII}}]\label{Auslander-Buchweitz}
  Let $F$ be an additive sub-bifunctor of $\Ext^1_{\L}(-,-)$, and $\X$
  an $F$-resolving subcategory of $\mod\L$.  Assume that the exact
  category $(\X,F)$ given in \textup{Proposition \ref{from F to exact}} has
  enough $F$-injectives and $\widehat{\X}=\mod\L$.
  Then the following assertions hold.
\begin{enumerate}[\rm(a)]
\item $\X$ is a contravariantly finite subcategory of $\mod\L$ and
  $\Y:=\X^{\perp_F}$ is a covariantly finite subcategory of $\mod\L$. 
\item $\X$-$\resdim_{F}(\mod\L)=\id_{F}(\X^{\perp_F})$ holds.
\end{enumerate}
\end{prop}

\section{Elementary properties of $n$-precluster tilting subcategories}\label{section:3}

Recall that for $n\geq 1$ a subcategory $\C$ of $\mod\L$ is
called \emph{$n$-cluster tilting} if $\C$ is functorially finite and
\[\C = {^{\perp_{n-1}}}\C = \C^{\perp_{n-1}},\]
where $^{\perp_{n-1}}\C$ is the full subcategory of $\mod\L$ given by
the modules 
\[{}^{\perp_{n-1}}\C=\{ X\in\mod\L \mid \Ext^i_\L(X,\C)=0 \text{\ for\ } 0 < i < n\}.\]
The full subcategory $\C^{\perp_{n-1}}$ is defined dually. 
In particular, it follows immediately from the definition that $\C$ is
a generator-cogenerator for $\mod\L$ and $\Ext^i_\L(\C,\C)=0$ for $0<
i < n$. In \cite{I} the functors
\[\tau_n = \tau\Omega^{n-1}_\L\colon \umod\L\to \omod\L\ \mbox{ and }\ 
\tau_n^{-} = \tau^-\Omega^{-(n-1)}_\L\colon \omod\L\to \umod\L\]
are defined as the \emph{$n$-Auslander--Reiten translations.} 
By \cite[Theorem 2.3.1]{I3}, the pair $(\tau_n^-,\tau_n)$ forms an adjunction.
By \cite[Theorem 1.4.1]{I}, they induce equivalences
\begin{equation}\label{taun equivalences}
\tau_n\colon\underline{{}^{\perp_{n-1}}\L}\to\overline{D\L^{\perp_{n-1}}}\ \mbox{ and }\ 
\tau_n^-\colon\overline{D\L^{\perp_{n-1}}}\to\underline{{}^{\perp_{n-1}}\L}.
\end{equation}
In particular $\tau_n$ and $\tau_n^-$ give bijections between
indecomposable non-projective modules in ${}^{\perp_{n-1}}\L$ and
indecomposable non-injective modules in $D\L^{\perp_{n-1}}$, while
they do not preserve indecomposability for arbitrary modules.
Moreover, for any $n$-cluster tilting subcategory $\C$ of $\mod\L$,
they restrict to equivalences
\[\tau_n\colon\underline{\C}\to\overline{\C}\ \mbox{ and }\ 
\tau_n^-\colon\overline{\C}\to\underline{\C}.\]
The next result gives a higher analog of Auslander--Reiten duality.

\begin{lem}[\protect{\cite[Theorem 1.5]{I}}]\label{prop:tauextformula}
We have the following.
\begin{enumerate}[\rm(a)]
\item $\uHom_\L(C,A)\simeq D\Ext^{n}_\L(A,\tau_n(C))$ and
  $\Ext^i_\L(C,A)\simeq D\Ext^{n-i}_\L(A,\tau_n(C))$ for
  $0<i<n$, for all modules $C$ in $^{\perp_{n-1}}\L$ and all modules $A$
  in $\mod\L$.
\item $\oHom_\L(C,A)\simeq D\Ext^{n}_\L(\tau_n^-(A),C)$ and
  $\Ext^i_\L(C,A)\simeq D\Ext^{n-i}_\L(\tau_n^-(A),C)$ for
  $0<i<n$, for all modules $C$ in $\mod\L$ and all modules $A$ in
  $D\L^{\perp_{n-1}}$. 
\end{enumerate}
\end{lem}

\begin{proof}
We give a proof of the second isomorphisms in our language of relative homological algebra.

(a) Let $F=F^{\L}$ be an additive sub-bifunctor of
  $\Ext^1_\L(-,-)$.  Let $C$ be in $^{\perp_{n-1}}\L$, and let
\[\eta\colon 0\to \Omega_\L^{n-1}(C)\to P_{n-2}\to \cdots\to P_0\to C\to 0\]
be a minimal projective resolution of $C$. Since $C$ is in
$^{\perp_{n-1}}\L$, we have that the exact sequence $\eta$ is
$F$-exact. Using this we have for $0<i<n$ and an arbitrary $\L$-module $A$ that
\begin{align}
D\Ext^i_F(C,A) & \simeq D\Ext^1_F(\Omega_\L^{i-1}(C),A)\notag\\
 & \simeq D\oHom_{F}(\Omega^i_\L(C),A)\notag\\
 & \simeq \Ext^1_F(\tau^-(A),\Omega_\L^i(C))\notag\\
 & \simeq \Ext^{n-i}_F(\tau^-(A),\Omega_\L^{n-1}(C))\notag\\
 & \simeq \Ext^{n-i}_{\L}(A,\tau_n(C)),\notag
\end{align}
where the second isomorphism uses that $P_i$ is in $\I(F)$, the third
is the Auslander--Reiten formula, the first and the fourth use dimension shift and the
last one is given by Theorem \ref{prop:exttauisom} and $G = F_\L=\Ext^1_\L(-,-)$.

Since $C$ is in $^{\perp_{n-1}}\L$, we have that $\Ext^i_\L(C,A)\simeq
\Ext^i_F(C,A)$ for $0<i<n$ and for all modules $A$ in $\mod\L$ by
Proposition \ref{prop:relative=absolute}(a). Thus the assertion follows.

(b) Similar proof as in (a). 
\end{proof}

We introduce the notion of an $n$-precluster tilting subcategory
by relaxing the conditions for an $n$-cluster tilting subcategory.

\begin{defin}\label{defin:n-precluster}
A subcategory $\C$ of $\mod\L$ is called an \emph{$n$-precluster tilting subcategory} if it
satisfies the following conditions.
\begin{enumerate}[(i)]
\item $\C$ is a generator-cogenerator for $\mod\L$,
\item $\tau_n(\C) \subseteq \C$ and $\tau_n^-(\C)\subseteq \C$, 
\item $\Ext^i_\L(\C,\C) = 0$ for $0 < i < n$,  
\item $\C$ is a functorially finite subcategory of $\mod\L$.  
\end{enumerate}
If moreover $\C$ admits an additive generator $M$,
we say that $M$ is an \emph{$n$-precluster tilting module}.
\end{defin}

Using $\tau_n$ and $\tau_n^-$, we define the subcategories
\[\P_n=\add\{\tau_n^{-i}(\L)\}_{i=0}^\infty\ \mbox{ and }\ \I_n=\add\{\tau_n^{i}(D(\L_\L))\}_{i=0}^\infty.\]
For any $n$-precluster tilting subcategory $\C$ of $\mod\L$, we have
\[\P_n\vee\I_n\subseteq\C\ \mbox{ and }\ \C\subseteq D\L^{\perp_{n-1}}\cap{}^{\perp_{n-1}}\L\]
by (i), (ii) and (i), (iii) respectively.

Recall from \cite{ASIV} that an Artin algebra $\L$ is called
\emph{$\tau$-selfinjective} if $\P_1$ is of finite type, which is
shown to be equivalent to that $\P_1$ is equal to $\I_1$.
We show next that this is equivalent to the existence of a $1$-precluster tilting
$\L$-module.  

\begin{example}
  An Artin algebra $\L$ is $\tau$-selfinjective if and only if $\L$
  has a $1$-precluster tilting module.
\end{example} 

\begin{proof}
  If $\L$ is $\tau$-selfinjective, then clearly the additive generator of
  $\P_1$ is a $1$-precluster tilting module since $D\L\in\I_1=\P_1$.  If
  $\L$ has a finite $1$-precluster subcategory $\C$, then it is clear
  from the definition that $\P_1$ is contained in $\C$.  Hence
  $\P_1$ is of finite type and $\L$ is $\tau$-selfinjective.
\end{proof}

This observation leads us to the following definition.

\begin{defin}\label{taun selfinjective}
  An Artin algebra $\L$ is called \emph{$\tau_n$-selfinjective} if
  $\L$ admits an $n$-precluster tilting module.
\end{defin}

This is a common generalization of selfinjective algebras and
$n$-representation-finite algebras.  We continue by asking and giving
one answer to the natural question: When is an Artin algebra
$\tau_n$-selfinjective?

\begin{prop}\label{prop:existenceofnprecluster}
Let $\L$ be an Artin algebra and $n\geq 1$.  Then the following conditions are equivalent. 
\begin{enumerate}[\rm(i)]
\item $\L$ is $\tau_n$-selfinjective.
\item $\P_n\vee \I_n$ is of finite type and $\Ext^i_\L(\P_n \vee \I_n,\P_n \vee \I_n)=0$ for $0<i<n$. 
\item $\I_n$ is of finite type, $\I_n\subset{}^{\perp_{n-1}}\L$ and $\Ext^i_\L(\I_n,\I_n)=0$ for $0<i<n$. 
\item $\L\in\I_n$ and $\Ext^i_\L(\I_n,\I_n)=0$ for $0<i<n$. 
\item $\P_n$ is of finite type, $\P_n\subset D\L^{\perp_{n-1}}$ and $\Ext^i_\L(\P_n,\P_n)=0$ for $0<i<n$. 
\item $D\L\in\P_n$ and $\Ext^i_\L(\P_n,\P_n)=0$ for $0<i<n$. 
\end{enumerate}
Moreover, if these conditions are satisfied, then every additive generator of $\P_n\vee\I_n$ is
an $n$-precluster tilting $\L$-module.
\end{prop}
\begin{proof} (i) is equivalent to (ii): Assume that there exists an
  $n$-precluster tilting module $M$ in $\mod\L$. Since
  $\P_n\vee\I_n\subset\add M$, it is immediate that $\P_n \vee \I_n$
  is of finite type and satisfies $\Ext^i_\L(\P_n \vee \I_n,\P_n \vee
  \I_n)=0$ for $0<i<n$.  Conversely, if $\P_n \vee \I_n$ satisfies
  (ii), then an additive generator of $\P_n\vee\I_n$ is an
  $n$-precluster tilting $\L$-module.

(ii) implies (iii): This is immediate. 

(iii) implies (iv): For each indecomposable non-projective module
$X$ in $^{\perp_{n-1}}\L$, we know that $\tau_n(X)$ is indecomposable again
by the equivalence \eqref{taun equivalences}.
Since $\I_n$ is of finite type, then $\tau_n^l(I)\neq 0$ is projective for some
$l\geq 0$ for all indecomposable injective modules $I$. Since the number of
indecomposable projectives and of indecomposable injectives coincide, all
indecomposable projective modules must occur in this
way, hence $\L$ is in $\I_n$ and (iv) is satisfied. 

(iv) implies (ii): For each indecomposable projective modules
$P$ there exists an indecomposable injective module $I$ such
that $P\simeq \tau_n^l(I)$. Since $\tau_n^{-i}(P)\simeq\tau_n^{l-i}(I)$
for $0\le i\le l$ and $\tau_n^{-(l+1)}(P)=0$ hold, we have
$\P_n\vee\I_n=\I_n$. Thus (ii) is satisfied.

The equivalences of (ii), (v) and (vi) are shown dually.
\end{proof}

Using the above we have the following consequence for $n=2$
thanks to a general property of $\I_2$.

\begin{prop}[\protect{cf.\ \cite[Proposition 1.7]{I4}}]
Let $\L$ be an Artin algebra. Then $\L$ is $\tau_2$-selfinjective if and only if $\L\in\I_2$. 
\end{prop}

\begin{proof}
  By \cite[Proposition 2.5]{I4} the subcategory $\I_2$ satisfies $\Ext_\L^1(\I_2,\I_2)=0$.
  The claim then follows
  immediately from Proposition \ref{prop:existenceofnprecluster}(iv)$\Rightarrow$(i). 
\end{proof}

The next results give analog of properties of higher Auslander--Reiten translation
for $n$-cluster tilting subcategories.

\begin{prop}\label{lem:preclustercategory}
Let $\L$ be an Artin algebra and $\C$ an $n$-precluster tilting
subcategory with $n\ge1$. Then we have the following.
\begin{enumerate}[\rm(a)]
\item We have mutually quasi-inverse equivalences
$\tau_n\colon\underline{\C}\to\overline{\C}$
and $\tau_n^-\colon\overline{\C}\to\underline{\C}$.
\item $\tau_n$ and $\tau_n^-$ give bijections between indecomposable
  non-projective modules in $\C$ and indecomposable non-injective
  modules in $\C$.
\item We have $\add\{\tau^-_{n}(\C), \L\} = \C=\add\{\tau_{n}(\C), D\L\} $. 
\item There exists a full subcategory $\D$ of $\mod\L$ such that
  $\C=\add\{\P_n\vee\I_n,\D\}$, $(\P_n\vee\I_n)\cap\D=\{0\}$ and
  $\tau_n(\D)=\D=\tau_n^-(\D)$.
\end{enumerate}
\end{prop}

\begin{proof}
  (a) We have $\C\subseteq {D\L^{\perp_{n-1}}}\cap{}^{\perp_{n-1}}\L$.
  Since $\tau_n(\C)\subseteq\C$ and $\tau_n^-(\C)\subseteq\C$ hold,
  the claim follows from the equivalences \eqref{taun equivalences}.

  (b)(c)(d) These follow immediately from (a).
\end{proof}

An $n$-cluster tilting subcategory $\C$ satisfies by definition the
equalities $\C = {^{\perp_{n-1}}}\C = \C^{\perp_{n-1}}$. The same is
not true for an $n$-precluster tilting subcategory $\C$ of $\mod\L$
(e.g.\ $\L$ is non-semisimple selfinjective and $\C=\add\L$).
But the equality ${^{\perp_{n-1}}}\C = \C^{\perp_{n-1}}$ is shown
still to be true.

\begin{prop}\label{left and right orthogonal}
  Let $\C$ be a subcategory of $\mod\L$ satisfying the conditions
  \textup{(i)}, \textup{(iii)} and \textup{(iv)} in \textup{Definition \ref{defin:n-precluster}}.
  \begin{enumerate}[\rm(a)]
  \item We have $\Omega^-_{F^{\C}}({}^{\perp_{n-1}}\C)\subseteq
  {}^{\perp_{n-1}}\C$ and $\Omega_{F_{\C}}(\C^{\perp_{n-1}})\subseteq\C^{\perp_{n-1}}$.
  \item Assume $n>1$. Then $\C$ is an $n$-precluster tilting subcategory of $\mod\L$ if and only if
  $\C^{\perp_{n-1}} = {^{\perp_{n-1}}\C}$.
  \end{enumerate}
\end{prop}

\begin{proof}
  (a) We only prove the first inclusion since the other one is dual.
  
  Let $X$ be in $^{\perp_{n-1}}\C$, and
  $\eta\colon 0\to X\extto{f}C^0\to Y\to 0$ be an exact sequence with
  a minimal left $\C$-approximation $f$ of $X$. Applying $\Hom_\L(-,\C)$,
  one easily shows that $Y$ is in $^{\perp_{n-1}}\C$.   
  
  (b) It suffices to show that $\tau_n(\C)\subseteq \C$ and $\tau_n^-(\C)\subseteq \C$
  hold if and only if $\C^{\perp_{n-1}} = {^{\perp_{n-1}}\C}$ holds.
   
  Assume $\tau_n(\C)\subseteq \C$ and $\tau_n^-(\C)\subseteq \C$.
  Since $\C\subset{}^{\perp_{n-1}}\L$, we have
  \[\Ext^i_\L(\C,{}^{\perp_{n-1}}\C)\simeq D\Ext^{n-i}_\L({}^{\perp_{n-1}}\C,\tau_n(\C))=0\]
  for all $0<i<n$ by Lemma \ref{prop:tauextformula}.  Thus
  $^{\perp_{n-1}}\C \subseteq\C^{\perp_{n-1}}$ holds.  Similarly
  $\C^{\perp_{n-1}}\subseteq {^{\perp_{n-1}}\C}$ holds.  Consequently
  $\C^{\perp_{n-1}} = {^{\perp_{n-1}}\C}$.

  Assume that $\C^{\perp_{n-1}} = {^{\perp_{n-1}}\C}$.  By Lemma
  \ref{prop:tauextformula}, we have
  \[\Ext^i_\L({}^{\perp_{n-1}}\C,\tau_n(\C))\simeq D\Ext^{n-i}_\L(\C,{}^{\perp_{n-1}}\C)=D\Ext^{n-i}_\L(\C,\C^{\perp_{n-1}})=0\]
  for all $0<i<n$. Thus
  $\tau_n(\C)\subseteq(^{\perp_{n-1}}\C)^{\perp_{n-1}}$.  We show
  $(^{\perp_{n-1}}\C)^{\perp_{n-1}}\subseteq\C$.  Since
  $\C\subseteq{}^{\perp_{n-1}}\C$, we have
  ${}^{\perp_{n-1}}\C=\C^{\perp_{n-1}}\supseteq(^{\perp_{n-1}}\C)^{\perp_{n-1}}$.
  For any $X$ in $(^{\perp_{n-1}}\C)^{\perp_{n-1}}$, there exists an
  exact sequence $\eta\colon 0\to X\to C^0\to Y\to 0$ with $C^0\in\C$
  and $Y\in{}^{\perp_{n-1}}\C$ by (a). This splits by the assumption $n>1$, and $X$ is a direct
  summand of $C^0$. Hence $X$ is in $\C$.
  
  Similarly we prove that $\tau_n^-(\C)\subseteq \C$.
\end{proof}

Now we introduce the following category $\U(\C)$, which is an analog
of Calabi--Yau reduction of triangulated categories \cite{IY}.

\begin{defin}
For an $n$-precluster tilting subcategory $\C$ in $\mod\L$, let
\[\Z(\C)=\C^{\perp_{n-1}}={}^{\perp_{n-1}}\C\ \mbox{ and }\ \U(\C)=\Z(\C)/[\C].\]
\end{defin}

Note that when $\C$ is a $1$-precluster tilting subcategory, then
$\Z(\C)=\mod\L$, since the orthogonality condition is void.

The next result gives basic properties of $\Z(\C)$ which generalize
those of $n$-cluster tilting subcategories \cite[Theorems 2.3, 2.3.1,
2.2.3]{I}.  In particular it gives higher Auslander--Reiten
translation for $\Z(\C)$ extending Proposition
\ref{lem:preclustercategory}.

\begin{thm}\label{Z(C)}
Let $\L$ be an Artin algebra and $\C$ an $n$-precluster tilting subcategory of
$\mod\L$ with $n\ge1$.  Then we have the following.
\begin{enumerate}[\rm(a)]
\item We have equivalences $\tau_n\colon\underline{\Z(\C)}\to\overline{\Z(\C)}$
and $\tau_n^-\colon\overline{\Z(\C)}\to\underline{\Z(\C)}$.
\item $\tau_n$ and $\tau_n^-$ give bijections between indecomposable
non-projective modules in $\Z(\C)$ and indecomposable non-injective modules
in $\Z(\C)$.
\item We have $\add\{\tau_n(\Z(\C)), D\L\} = \Z(\C)=\add\{\tau^-_n(\Z(\C)), \L\}$.
\item For every $X$ and $Y$ in $\Z(\C)$ and $0<i<n$, we have functorial 
isomorphisms
\begin{eqnarray*}
\uHom_\L(X,Y)\simeq D\Ext^{n}_\L(Y,\tau_n(X)),&&\Ext^i_\L(X,Y)\simeq D\Ext^{n-i}_\L(Y,\tau_n(X))\\
\oHom_\L(X,Y)\simeq D\Ext^{n}_\L(\tau_n^-(Y),X),&&\Ext^i_\L(X,Y)\simeq D\Ext^{n-i}_\L(\tau_n^-(Y),X).
\end{eqnarray*}
\item For every $X$ in $\mod\L$, there exists an $F_{\C}$-exact sequence
\[0\to Z_{n-1}\to C_{n-2}\to\cdots\to C_0\to X\to0\]
with $C_i$ in $\C$ for every $i$ and $Z_{n-1}$ in $\Z(\C)$.
\end{enumerate}
\end{thm}

\begin{proof}
  (a) Thanks to the equivalences \eqref{taun equivalences}, it
  suffices to show $\tau_n(\Z(\C))\subseteq\Z(\C)$ and
  $\tau_n^-(\Z(\C))\subseteq\Z(\C)$. Using Lemma
  \ref{prop:tauextformula}, we have
\begin{eqnarray*}
0 = D\Ext^i_\L({\Z(\C)},\C) \simeq
\Ext^{n-i}_\L(\C,\tau_n({\Z(\C)}))\\
0 = D\Ext^i_\L(\C,\Z(\C)) \simeq
\Ext^{n-i}_\L(\tau^-_n(\Z(\C)),\C)
\end{eqnarray*}
for $0 < i < n$. Thus the assertion follows.

(b)(c) Immediate from (a).

(d) This follows from Lemma \ref{prop:tauextformula}.

(e) Let $X$ be in $\mod\L$, and let $0\to\Omega_{F_{\C}}(X)\to C_0\to X\to 0$ be an
$F_{\C}$-exact sequence given by an $F_{\C}$-projective cover.
Then $\Ext^1_\L(\C,\Omega_{F_{\C}}(X))=0$ holds. Taking an $F_{\C}$-exact sequence
$0\to \Omega_{F_{\C}}^2(X)\to C_1\to \Omega_{F_{\C}}(X)\to0$
given by an $F_{\C}$-projective cover, it follows that
$\Ext^i_\L(\C,\Omega_{F_{\C}}^2(X))=0$ for $i=1,2$.
Continuing this process we obtain that $\Ext^i_\L(\C,\Omega^{n-1}_{F_{\C}}(X))=0$ for
$0<i<n$, and hence $\Omega^{n-1}_{F_{\C}}(X)$ is in $\Z(\C)$.
\end{proof}

The following easy property below is useful.

\begin{lem}\label{FC=FC}
Let $\C$ be an $n$-precluster tilting subcategory of $\mod\L$ with $n\ge1$.
\begin{enumerate}[\rm(a)]
\item We have $F_{\C}|_{\Z(\C)\times\Z(\C)}=F^{\C}|_{\Z(\C)\times\Z(\C)}$.
\item For every $0<i<n$, we have
$\Ext^i_{F_{\C}}(-,-)|_{\Z(\C)\times\Z(\C)}=\Ext^i_\L(-,-)|_{\Z(\C)\times\Z(\C)}=
\Ext^i_{F^{\C}}(-,-)|_{\Z(\C)\times\Z(\C)}$.
\item The exact categories $(\Z(\C),F_{\C})$ and $(\Z(\C),F^{\C})$ are the same. It coincides with $(\Z(\C),\Ext^1_\L(-,-))$ if $n\ge 2$.
\end{enumerate}
\end{lem}

\begin{proof}
  (a) (c) For $n=1$, we have $F_{\C}=F^{\tau\C}=F^{\C}$ by
  \eqref{defect formula} since $\C$ is 1-precluster tilting, and for
  $n\ge2$, we have
  $F_{\C}|_{\Z(\C)\times\Z(\C)}=\Ext^1_{\L}(-,-)|_{\Z(\C)\times\Z(\C)}=F^{\C}|_{\Z(\C)\times\Z(\C)}$.

(b) This is immediate from Proposition \ref{prop:relative=absolute}.
\end{proof}

The category $\Z(\C)$ enjoys the following remarkable properties.

\begin{prop}\label{Z(C) is Frobenius}
  Let $\C$ be an $n$-precluster tilting subcategory of $\mod\L$ for
  some $n\geq 1$.
\begin{enumerate}[\rm(a)]
\item $\Z(\C)$ is extension closed.
\item $\Z(\C)$ has a structure of a Frobenius category whose short
  exact sequences are precisely $F_{\C}$-exact sequences,
  and projective-injective objects are precisely $\C$.
\item $\U(\C)$ has a structure of a triangulated category with the
  suspension functor $[1]=\Omega_{F^{\C}}^{-1}$.
\end{enumerate}
\end{prop}

\begin{proof}
(a) This is easily checked by using the long exact sequence of $\Ext$'s.

(b) By Proposition \ref{from F to exact}, we have an exact category
$(\Z(\C),F_{\C})$, which coincides with $(\Z(\C),F^{\C})$ by Lemma \ref{FC=FC}(c).
Thus any object in $\C$ is projective-injective in $\Z(\C)$.  By Proposition
\ref{left and right orthogonal}(a), $\Z(\C)$ has enough projectives and enough
injectives.  Therefore the projective and the injective objects
coincide, and they are equal to $\C$.  Thus the assertion follows.

(c) This is a general property of Frobenius categories \cite{Ha}.
\end{proof}

We show that the triangulated category $\U(\C)$ admits a Serre functor,
which is an analog of \cite[Theorem 4.7]{IY}.

\begin{thm}\label{Serre functor}
Let $\C$ be an $n$-precluster tilting subcategory of $\mod\L$ with $n\ge1$.
\begin{enumerate}[\rm(a)]
\item The triangulated category $\U(\C)$ admits a Serre functor $S$
  given by $S=[n]\circ\tau_n$.
\item The triangulated category $\U(\C)$ has almost split triangles,
  i.e.\ any indecomposable object $X$ in $\U(\C)$ has almost split
  triangles in $\U(\C)$:
\[SX[-1]\to E\to X\to SX\ \mbox{ and }\ X\to E'\to S^{-1}X[1]\to X[1].\]
\item The Frobenius category $\Z(\C)$ has almost split sequences,
  i.e.\ any indecomposable module $X$ in $\Z(\C)\setminus\C$ has almost
  split sequences in $\Z(\C)$
\[0\to\tau_{\Z(\C)}(X)\to E\to X\to0\ \mbox{ and }\ 0\to X\to E'\to\tau^-_{\Z(\C)}(X)\to0,\]
where $\tau_{\Z(\C)}:=\Omega_{\Z(\C)}^{-(n-1)}\tau_n$ and $\tau_{\Z(\C)}^-:=\Omega_{\Z(\C)}^{n-1}\tau_n^-$.
\end{enumerate}
\end{thm}
\begin{proof}
(a) Let $F=F^{\C}$ and $G=F_{\C}$. Then $\I(F)=\P(G)=\C$ holds.
For all $X$ in $\Z(\C)$, take an $F$-injective coresolution
\[0\to \tau_n(X)\to C^0\to\cdots\to C^{n-1}\to\Omega^{-n}_F(\tau_n(X))\to 0,\]
which gives a $G$-projective resolution since $\Z(\C)$ is Frobenius
by Proposition \ref{Z(C) is Frobenius}.
Thus, for every $Y$ in $\mod\L$, we obtain functorial isomorphisms
\begin{align}
\Ext^n_G(Y,\tau_n(X))&\simeq
\Coker(\Hom_\L(Y,C^{n-1})\to\Hom_\L(Y,\Omega^{-n}_F(\tau_n(X)))\notag\\
&\simeq\uHom_G(Y,\Omega^{-n}_F(\tau_n(X))).\notag
\end{align}
Moreover we have functorial isomorphisms
\begin{align}
\Ext^n_G(Y,\tau_n(X))& \simeq \Ext^n_F(\tau^-(Y),\Omega_\L^{n-1}(X))\notag\\
& \simeq \Ext^1_F(\tau^-(Y),X)\notag\\
& \simeq D\oHom_{F}(X,Y),\notag
\end{align}
where the first isomorphism is given by Theorem \ref{prop:exttauisom},
the second follows from $\P(\L)\subset\I(F)$,
and the third is the Auslander--Reiten duality (Proposition \ref{prop:AR-formula}).
Combining them, for every $Y$ in $\Z(\C)$, we have functorial isomorphisms
\begin{align}
\Hom_{\U(\C)}(Y,\Omega_F^{-n}(\tau_n(X)))&=\uHom_G(Y,\Omega^{-n}_F(\tau_n(X)))\notag\\
&\simeq D\oHom_{F}(X,Y)=D(\Hom_{\U(\C)}(X,Y)).
\end{align}
Thus $\U(\C)$ has $S=\Omega_F^{-n}\tau_n=[n]\circ\tau_n$ as a Serre functor. 

(b) Since $\U(\C)$ has a Serre functor, it has almost split triangles \cite{RV}.

(c) Immediate from (b).
\end{proof}

The next bijective correspondence is an analog of a property of
Calabi--Yau reduction of triangulated categories \cite[Theorem 4.9]{IY}.
One of the consequences is that, if $\P_n\vee\I_n$ is a functorially finite subcategory
of $\mod\L$, then the classification problem of $n$-cluster tilting subcategories in $\mod\L$
can be reduced to the same problem in the triangulated category $\U(\P_n\vee\I_n)$.

\begin{thm}
Let $\C$ be an $n$-precluster tilting subcategory of $\mod\L$ with $n\ge1$.
Then there exists a bijection between $n$-cluster tilting subcategories of $\mod\L$
containing $\C$ and $n$-cluster tilting subcategories of $\U(\C)$ given by
$\C'\mapsto\C'/[\C]$.
\end{thm}

\begin{proof}
Any $n$-cluster tilting subcategory of $\mod\L$ containing $\C$ is clearly contained in
$\Z(\C)$. On the other hand, let $\C'$ be a subcategory of $\mod\L$ containing $\C$.
It follows from Lemma \ref{FC=FC}(b) that $\C'$ is an $n$-cluster tilting subcategory
of $\mod\L$ if and only if $\C'/[\C]$ is an $n$-cluster tilting subcategory of $\U(\C)$.
Thus the assertion follows.
\end{proof}

Next we give another proof of existence of almost split sequences in $\Z(\C)$
by showing that $\Z(\C)$ is functorially finite. 

\begin{thm}\label{prop:Dperpfunctfinite}
Let $\C$ be an $n$-precluster tilting subcategory of $\mod\L$
with $n\geq 1$.
\begin{enumerate}[\rm(a)]
\item $\Z(\C)$ is $F_\C$-resolving and $F^\C$-coresolving in $\mod\L$.
\item $\Z(\C)$ is functorially finite in $\mod\L$
with $\Z(\C)$-$\resdim_{F_\C}(\mod\L) \le n-1$ and $\id_{F_\C}\C \le n-1$.
\item $\Z(\C)$ has almost split sequences.
\end{enumerate}
\end{thm}

\begin{proof}
Since the case $n=1$ is clear, we assume $n>1$ in the rest.

(a) This is easily checked by using the long exact sequence of $\Ext$'s.

(b) By Theorem \ref{Z(C)}(e) we have $\Z(\C)$-$\resdim_{F_\C}(\mod\L)
\leq n-1$ and hence $\widehat{\Z(\C)} = \mod\L$.  Moreover
$\Z(\C)$ is $F_\C$-resolving by (a), and the exact category
$(\Z(\C),F_\C)$ has enough injectives $\C$ by Proposition \ref{Z(C) is
  Frobenius}(b). Therefore, by Proposition \ref{Auslander-Buchweitz},
$\Z(\C)$ is contravariantly finite in $\mod\L$ and $\id_{F_\C}\C
\leq n-1$ holds.  Dual arguments show that $\Z(\C)$ is
covariantly finite, hence it is functorially finite.

(c) Since $\Z(\C)$ is extension closed and functorially finite, it has
almost split sequences \cite{ASm}.
\end{proof}

We end this section with the following observation on
Auslander--Buchweitz type approximations by $\Z(\C)$.

\begin{cor}\label{approximation sequence}
Let $\C$ be an $n$-precluster tilting subcategory of $\L$ with $n\ge1$ and $X$ in $\mod\L$.
\begin{enumerate}[\rm(a)]
\item For every $0\le i\le n-1$, there exists an $F_{\C}$-exact sequence
\begin{equation}\label{Z_i}
0\to C_{n-1}\xrightarrow{f_{n-1}}\cdots\xrightarrow{f_{i+2}}C_{i+1}\xrightarrow{f_{i+1}}Z_i\xrightarrow{f_i}
C_{i-1}\xrightarrow{f_{i-1}}\cdots\xrightarrow{f_1}C_0\xrightarrow{f_0}X\to0
\end{equation}
with $Z_i$ in $\Z(\C)$ and $C_j$ in $\C$ for every $j$. Moreover $\Im f_j$ is in $\C^{\perp_j}$ for every $j$.
\item For every $0\le i\le n-1$, there exists an $F^{\C}$-exact sequence
\[0\to X\xrightarrow{f^0}C^0\xrightarrow{f^1}\cdots\xrightarrow{f^{i-1}}C^{i-1}\xrightarrow{f^i}Z^i\xrightarrow{f^{i+1}}
C^{i+1}\xrightarrow{f^{i+2}}\cdots\xrightarrow{f^{n-1}}C^{n-1}\to0\]
with $Z^i$ in $\Z(\C)$ and $C^j$ in $\C$ for every $j$. Moreover $\Im f^j$ is in ${}^{\perp_j}\C$ for every $j$.
\end{enumerate}
\end{cor}

\begin{proof}
(a) The case $i=n-1$ was shown in Theorem \ref{Z(C)}(e).

Assume that the sequence \eqref{Z_i} exists.
Take an exact sequence $0\to Z_i\to C_i\to Z'_{i-1}\to0$ with $C_i$ in $\C$ and $Z'_{i-1}$ in $\Z(\C)$,
and consider the pushout diagram.
\[\xymatrix@C=1.2em@R=1em{
&&&&0\ar[d]&0\ar[d]\\
0\ar[r]&C_{n-1}\ar[r]\ar@{=}[d]&\cdots\ar[r]&C_{i+1}\ar[r]\ar@{=}[d]&Z_i\ar[r]\ar[d]&C_{i-1}\ar[r]\ar[d]&C_{i-2}\ar[r]\ar@{=}[d]&\cdots\ar[r]&C_0\ar[r]\ar@{=}[d]&X\ar[r]\ar@{=}[d]&0\\
0\ar[r]&C_{n-1}\ar[r]&\cdots\ar[r]&C_{i+1}\ar[r]&C_i\ar[r]\ar[d]&Z_{i-1}\ar[r]\ar[d]&C_{i-2}\ar[r]&\cdots\ar[r]&C_0\ar[r]&X\ar[r]&0\\
&&&&Z'_{i-1}\ar@{=}[r]\ar[d]&Z'_{i-1}\ar[d]\\
&&&&0&0
}\]
Since $Z_{i-1}\simeq Z'_{i-1}\oplus C_{i-1}$ belongs to $\Z(\C)$, the lower sequence gives the desired sequence for $i-1$.
Thus the assertion follows inductively.

(b) This is dual to (a).
\end{proof}

\section{Auslander--Solberg correspondence}\label{section:4}

In \cite{I3}, an Artin algebra $\G$ is called an \emph{$n$-Auslander algebra} if
$\domdim\G\geq n+1 \geq \gldim\G$. In addition,
there exists a bijection between the set of equivalence classes of
$n$-cluster tilting subcategories with additive generators and
Morita-equivalence classes of $n$-Auslander algebras. A slightly more
general class of Artin algebras were considered in \cite{ASIV} for
$n=1$, namely Artin Gorenstein algebras $\G$ with $\domdim\G=2=\id
{_\G\G}$. These algebras were classified in \cite{ASIV}
in terms of $\tau$-selfinjective algebras $\L$.  
If one considers the weaker condition $\domdim\G\geq 2\geq \id {_\G\G}$,
then the difference is precisely the class of selfinjective algebras
(Proposition \ref{basic properties of minimal AG}).

This section is devoted to extending the above result by exhibiting a
bijection between the Morita-equivalence classes of Artin Gorenstein
algebras $\G$ with $\domdim\G\geq n+1\geq \id {_\G\G}$ for
$n\geq 1$ and equivalence classes of finite $n$-precluster
tilting subcategories.  This class of Gorenstein algebras we call
\emph{$n$-minimal Auslander--Goren\-stein algebras}.

First we give some general properties of the class of
$n$-minimal Auslander--Gorenstein algebras. 

\begin{prop}\label{basic properties of minimal AG}
\begin{enumerate}[\rm(a)]
\item \sloppypar If $\G$ is an $n$-minimal Auslander--Gorenstein algebra, then
  either $\domdim\G=n + 1 =\id_\G\G$ holds or $\G$ is
  selfinjective. In the latter case $\G$ is
  an $m$-Auslander--Gorenstein algebra for all $m\geq 0$.
\item An algebra $\G$ is $n$-minimal Auslander--Gorenstein if and only if
  $\G^\op$ is $n$-minimal Auslander--Gorenstein.
\end{enumerate}
\end{prop}

\begin{proof}
(a) If $\G$ is not selfinjective, then the last term of the minimal
injective coresolution of  $\G$ is not projective.
Thus the assertion follows from the definition of an
$n$-minimal Auslander--Gorenstein algebra.

(b) Although the assertion follows from well-known results
(\cite[Theorem 7.7]{T} and \cite[Corollary 5.5]{AR}), we give a direct argument here.
Assume that $\G$ is $n$-minimal Auslander--Gorenstein.
The minimal injective coresolution
\begin{equation}\label{GI}
0\to\G\to I^0\to\cdots\to I^n\to I^{n+1}\to0
\end{equation}
of $\G$ gives a projective resolution of the injective $\G$-module
$I^{n+1}$.  Thus $\Omega^{n+1}$ and $\Omega^{-(n+1)}$ give mutually
inverse bijections between indecomposable non-projective injective
$\G$-modules and indecomposable non-injective projective $\G$-modules
since their numbers are the same.  In particular, all indecomposable
non-projective injective $\G$-modules appear in $I^{n+1}$ as a direct
summand.  Applying $D$ to \eqref{GI}, we have an exact sequence
\[0\to DI^{n+1}\to DI^n\to\cdots\to DI^0\to D\G\to0\]
of $\G^{\op}$-modules with projective-injective $\G^{\op}$-modules
$DI^i$ with $0\le i\le n$.  Since all indecomposable non-injective
projective $\G^{\op}$-modules appear in $DI^{n+1}$ as a direct
summand, $\G^{\op}$ is also $n$-minimal Auslander--Gorenstein.
\end{proof}

Let us recall the following simple observation.

\begin{lem}[\protect{\cite[Lemma 3]{M}, \cite[$d=m=0$ in Theorem 4.2.1]{I3}}]\label{domdim}
  Let $\L$ be an Artin algebra, $M$ a generator-cogenerator of $\L$,
  $\G=\End_\L(M)$ and $n\ge1$. Then $\domdim\G\ge n+1$ holds if and
  only if $\Ext^i_{\L}(M,M)=0$ for $0<i<n$.
\end{lem}

The following result shows that a finite $n$-precluster tilting
subcategory gives rise to an $n$-minimal Auslander--Gorenstein
algebra. 

\begin{prop}\label{prop:endoring}
  Let $\L$ be an Artin algebra and $M$ an $n$-precluster tilting
  $\L$-module with $n\geq1$.
  \begin{enumerate}[\rm(a)] 
  \item $\G=\End_\L(M)$ is an
  $n$-minimal Auslander--Gorenstein algebra and the $\G$-module
  $I={}_{\G}M$ satisfies $\P(\G)\cap\I(\G)=\add I$ and
  $\L\simeq\End_\G(I)$.
  \item $M$ is an $F_M$-cotilting $\L$-module with $\id_{F_M}M\leq n-1$.
  \item $M$ is a projective $\L$-module if
  and only if $\G$ is selfinjective.
  \end{enumerate}
\end{prop}
\begin{proof}
  By Lemma \ref{domdim} we know that $\domdim \G \geq n+1$.  Let
  $G=F_{ M}\subseteq \Ext^1_\L(-,-)$. Then clearly
  $\Ext^i_G(M,M)=0$ for all $i>0$.  
Since $\tau^-_n(M)\in\add M$ by our assumption, we have
\begin{equation}\label{describe I(G)}
\I(G)=\add\{\tau(M),D\L\}=\add\{\Omega_\L^{-(n-1)}(M),D\L\}
\end{equation}
by Proposition \ref{lem:preclustercategory}(a).
The start of the minimal injective coresolution
\[0\to M\to I^0\to I^1\to \cdots \to I^{n-2}\to
\Omega^{-(n-1)}_\L(M)\to 0\] 
of $M$ is $G$-exact since $\Ext^i_\L(M,M)=0$ holds for every $0<i<n$.  It follows from
\eqref{describe I(G)} that $\id_G M\leq n-1$ and
$\I(G)\subseteq\widehat{\add M}$.  Therefore $M$ is a $G$-cotilting
module.  By Theorem \ref{thm:relativecotilting}(d), $\Hom_\L(M,M)=\G$
  is a cotilting $\G$-module with $\id {_\G \G}\leq \id_G M + 2
\leq n + 1$.  Thus $\G$ is an $n$-minimal Auslander--Gorenstein algebra.

Since $I=\Hom_{\L}(\L,M)$ belongs to $\P(\G)$ and
$DI=\Hom_{\L}(M,D\L)$ belongs to $\P(\G^{\op})$, we have
$I\in\P(\G)\cap\I(\G)$.  Taking a projective cover $P\to M\to0$ in
$\mod\L$ and applying $\Hom_\L(-,M)$, we have an exact sequence
$0\to\G\to\Hom_\L(P,M)$ in $\mod\G$ with $\Hom_\L(P,M)\in\add I$.
Thus $I$ is an additive generator of $\P(\G)\cap\I(\G)$.  In addition,
we have $\L\simeq\End_\G(I)$ from Theorem
\ref{thm:relativecotilting}(a) since $I = M$ is a $G$-cotilting
module.

The last claim is easy and left to the reader.  This completes the proof. 
\end{proof}

Next we show the converse, namely $n$-minimal Auslander--Gorenstein algebras
$\G$ give rise to a finite $n$-precluster tilting subcategory.

\begin{prop}\label{prop:redtoendoring}
  Let $\G$ be an $n$-minimal Auslander--Gorenstein algebra for $n\geq 1$.
  Let $I$ be an additive generator of $\P(\G)\cap\I(\G)$,
  $\L=\End_\G(I)$ and $M={_\L I}$.
  \begin{enumerate}[\rm(a)]
  \item $M$ is an $n$-precluster
  tilting $\L$-module such that $\End_\L(M)\simeq\G$.  
  \item $M$ is an $F_M$-cotilting $\L$-module with $\id_{F_M}M\leq n-1$.
  \item $\G$
  is selfinjective if and only if $M$ is a projective $\L$-module.
  \end{enumerate}
\end{prop}
\begin{proof}
We can assume that $I$ is basic.
Since $\domdim\G\ge n+1\ge2$, the module $I$ is a dualizing
summand of the cotilting $\G$-module $\G$.  Thus the claim
$\G\simeq \End_\L(M)$ follows directly from Theorem
\ref{thm:dualizingsummand}(a).

In the rest, we show that $M$ is an $n$-precluster tilting $\L$-module.

(i) We show that $M$ is a generator-cogenerator for $\L$. 

Since the $\G$-module $I$ belongs to $\add\G$,
the $\L$-module $\L=\End_{\G}(I)$ belongs to $\add\Hom_\G(\G,I)=\add M$.
Thus $M$ is a generator of $\L$. Since the $\G$-module $I$ belongs to $\add D(\G_\G)$,
the $\L^{\op}$-module $\L=\End_{\G}(I)$ belongs to $\add\Hom_\G(I,D(\G))=\add D(M)$.
Thus $M$ is a cogenerator of $\L$. 

(ii) It follows from Lemma \ref{domdim} that $\Ext_\L^i(M,M)=0$ for $0<i<n$.

(iii) It remains to show that both $\tau^-_{n}(M)$ and $\tau_{n}(M)$ are in $\add M$.

Let $G=F_{M}\subseteq \Ext_\L^1(-,-)$. Then by Theorem
\ref{thm:dualizingsummand} we infer that $M$ is a $G$-cotilting module
with $\id_G M\leq \max\{\id {_\G \G}-2,0\}\leq n-1$.  By (ii),
the injective coresolution
\[0\to M\to I^0\to I^1\to \cdots \to I^{n-2}\to
\Omega_\L^{-(n-1)}(M)\to 0\] 
of the $\L$-module $M$ is $G$-exact, and gives the start of a
$G$-injective coresolution. Since $\id_G M\le n-1$, the module
$\Omega_\L^{-(n-1)}(M)$ is in $\I(G)=\add\{D\L, \tau(M)\}$.  Hence
$\tau_{n}^-(M)=\tau^-(\Omega_\L^{-(n-1)}(M))$ is in $\add M$.

On the other hand, since $\G^{\op}$ is an $n$-minimal Auslander--Gorenstein
algebra such that $\P(\G^{\op})\cap\I(\G^{\op})=\add DI$ and
$\End_{\G^{\op}}(DI)=\L^{\op}$, the $\L^{\op}$-module $\tau_{n}^-(DM)$
is in $\add DM$ by the same argument.  Thus the $\L$-module
$\tau_{n}(M)$ is in $\add {_\L M}$.  This completes the proof that $M$
is an $n$-precluster tilting $\L$-module.

The last claim is easy and left to the reader. This completes the proof.
\end{proof}

Now we address the bijectivity of the correspondence.

\begin{thm}\label{AS correspondence}
  Fix $n\geq 1$.  There is a bijection between Morita-equivalence
  classes of $n$-minimal Auslander--Gorenstein algebras and equivalence classes
  of finite $n$-precluster tilting subcategories $\C$ of Artin
  algebras, where the correspondences are given in
  \textup{Propositions \ref{prop:endoring}} and
  \textup{\ref{prop:redtoendoring}}.
\end{thm}

\begin{proof}
  The assertions follow from Propositions \ref{prop:endoring} and
  \ref{prop:redtoendoring}.
\end{proof}

\begin{remark}
Note that the bijection for $n$-Auslander algebras given in \cite{I3} is
dual to Theorem \ref{AS correspondence}.
\end{remark}

The category $\Z(\C)$ associated to a finite $n$-precluster tilting subcategory
$\C$ has the following interpretation in terms of
the corresponding $n$-minimal Auslander--Gorenstein algebra.

\begin{thm}\label{CM duality}
Given an Artin algebra $\L$ with a finite $n$-precluster
  tilting subcategory $\C=\add M$,  let $\G=\End_\L(M)$ be the
  corresponding $n$-minimal Auslander--Gorenstein algebra.  Then
  $\Z(\C)$ and $\CM\G$ are dual categories via the functors 
  $\Hom_\L(-,M)\colon\Z(\C)\to\CM\G$ and $\Hom_\G(-,M)\colon\CM\G\to\Z(\C)$.
Moreover they induce triangle equivalences between $\U(\C)$ and
$(\underline{\CM}\,\G)^{\op}$. 
\end{thm}
\begin{proof}
  Let $G=F_M$.  Then $M$ is a $G$-cotilting module with $\id_G M \le
  n-1$ by the proof of Proposition \ref{prop:endoring}(b), and we have a
  duality
  \[\Hom_\L(-,M)\colon {^{\perp_{G}} M} \to {^\perp \G}=\CM\G\] 
  by Theorem \ref{thm:relativecotilting}(e).  Let us prove
  $\Z(\C)={^{\perp_{G}} M}$. Since $\id_G M \le n-1$, it follows that
  $\Ext^i_G(-,M)=0$ for all $i\geq n$.  Furthermore, since $M$ is
  in $\Z(\C)$, we have that $\Ext_G^i(-,M)=\Ext^i_\L(-,M)$ for $0<i<n$
  by Proposition \ref{prop:relative=absolute}(b). Thus
  $\Z(\C)={^{\perp_{G}} M}$ holds.  Since $\C=\add M$ corresponds to
  $\P(\G)$ via the duality $\Hom_\L(-,M)$, the last claim follows
  immediately.
\end{proof}

The class of $n$-Auslander algebras were introduced in \cite{I3} as
the Artin algebras $\G$ with $\domdim\G \geq n + 1 \geq
\gldim\G$. Now we characterize this subclass of algebras within the class of
$n$-minimal Auslander--Gorenstein algebras. 
\begin{thm}
  Let $\L$ be an Artin algebra, $\C=\add M$ a finite $n$-precluster tilting $\L$-module
  with $n\ge1$ and $\G=\End_{\L}(M)$.
  Then the following are equivalent.
\begin{enumerate}[\rm(a)]
\item  $\G$ is an $n$-Auslander algebra. 
\item  $\gldim \G < \infty$. 
\item  $\CM\G = \P(\G)$.
\item  $\Z(\C) = \C$. 
\end{enumerate}
\end{thm}
\begin{proof}
(a) implies (b): This is obvious.

(b) implies (c): If $X$ in $\CM\G$ is non-projective, then
$\Ext^m_\G(X,\G)\neq0$ holds for $m:=\pd_\G X>0$, a contradiction.

(c) implies (a): For any $X$ in $\mod\G$, $\Omega^{n+1}_\G(X)$ belongs
to $\CM\G=\P(\G)$ by using dimension shifting and $\id_\G\G\le n+1$ in Proposition \ref{prop:endoring}(a).
Thus $\gldim\G\le n+1$.

(c) is equivalent to (d): This follows from Theorem \ref{CM duality}.
\end{proof}

For a general Artin Gorenstein algebra $\G$, the category of
maximal Cohen--Macaulay modules $\CM\G={^\perp \G}$ is an extension closed
functorially finite subcategory of $\mod\G$ \cite{AB}. Therefore the category
$\CM\G$ has minimal left (respectively right) almost split maps and
almost split sequences.  We denote by $\tau_{\CM\G}$ the
Auslander--Reiten translation in $\CM\G$.

We end this section with the following easy observations, which
compare almost split sequences in $\Z(\C)$ with those in $\CM\G$.

\begin{prop}
Let $\L$ be an Artin algebra, $\C=\add M$ a finite $n$-precluster tilting $\L$-module with $n\ge1$ and $\G=\End_\L(M)$.
\begin{enumerate}[\rm(a)]
\item A morphism $f\colon A\to B$ is (minimal) left almost split
  (respectively (minimal) right almost split) in $\CM\G$ if and only
  if $\Hom_\G(f,M)\colon \Hom_\G(B,M)\to \Hom_\G(A,M)$ is (minimal)
  right almost split (respectively (minimal) left almost split) in
  $\Z(\C)$.
\item An exact sequence $0\to A\to B\to C\to 0$ in $\CM\G$ is almost
  split in $\CM\G$ if and only if $0\to \Hom_\G(C,M)\to
  \Hom_\G(B,M)\to \Hom_\G(A,M)\to 0$ is almost split in $\Z(\C)$.
\item $\tau^-_{\Z(\C)}(\Hom_\G(C,M))\simeq \Hom_\G(\tau_{\CM\G}(C),M)$
  holds for every indecomposable module $C$ in $\CM\G\setminus\P(\G)$,
  and $\tau^{-}_{\CM\G}(\Hom_\L(X,M))\simeq
  \Hom_\L(\tau_{\Z(\C)}(X),M)$ holds for every indecomposable module $X$
  in $\Z(\C)\setminus\C$.
\end{enumerate}
\end{prop}

\begin{proof}
All assertions are immediate from Theorem \ref{CM duality}.
\end{proof}

\section{$n$-fold almost split extensions}\label{section:5} 

Higher Auslander--Reiten theory on $n$-cluster tilting subcategories
was introduced in \cite{I}.  An $n$-almost split sequence in an
$n$-cluster tilting subcategory $\C$ is an $n$-fold exact sequence
\[\eta\colon 0\to Y\to C_n\to \cdots \to C_1\to X\to 0\]
with indecomposable objects $X$, $Y$ in $\C$ and objects
$\{C_i\}_{i=1}^n$ in $\C$, and the sequences
\begin{eqnarray*}
&0\to \Hom_\L(\C,Y)\to \Hom_\L(\C,C_n)\to \cdots \to \Hom_\L(\C,C_1)\to
\rad_\L(\C,X)\to 0&\\
&0\to \Hom_\L(X,\C)\to \Hom_\L(C_1,\C)\to \cdots \to \Hom_\L(C_n,\C)\to
\rad_\L(Y,\C)\to 0&
\end{eqnarray*}
are exact. This is equivalent to the statement that $\eta$
represents an element in the socle of $\Ext^n_\L(X,Y)$ as an
$\End_\L(X)^\op$- or $\End_\L(Y)$-module. These properties play a key role to
generalize the notion of $n$-almost split sequences to more general
categories, and we make the following definition. 

\begin{defin}\label{def:nfoldalmostsplitextension}
Let $\L$ be an Artin algebra, $X$ and $Y$ be modules in $\mod\L$, and
$\eta$ a non-zero element in $\Ext^n_\L(X,Y)$.
\begin{enumerate}[\rm(a)]
\item We say that $\eta$ is a \emph{right $n$-fold almost split extension of $X$} if,
for every $Z$ in $\mod\L$ and a non-zero element $\xi$ in $\Ext^n_\L(X,Z)$,
there exists a morphism $f\colon Z\to Y$ such that $\eta=f\xi$, where $f\xi$ stands for $\Ext^n_\L(X,f)(\xi)$.
\item We say that $\eta$ is a \emph{left $n$-fold almost split extension of $Y$} if,
for every $Z$ in $\mod\L$ and a non-zero element $\xi$ in $\Ext^n_\L(Z,Y)$,
there exists a morphism $g\colon X\to Z$ such that $\eta=\xi g$, where $\xi g$ stands for $\Ext^n_\L(g,Y)(\xi)$.
\item We say that $\eta$ is an \emph{$n$-fold almost split extension} if it is
a right $n$-fold almost split extension of $X$ and a left $n$-fold almost split
extension of $Y$.
\end{enumerate}
\end{defin}

We changed the terminology from $n$-almost split sequence to $n$-fold almost
split extension, since an element in $\Ext^n_\L(X,Y)$ can possibly be
represented by several different long exact sequences.

We need the following isomorphisms to study $n$-fold almost split extensions.
\begin{equation}\label{ExtC}
\Ext^n_\L(X,-)=\Ext^1_\L(\Omega_\L^{n-1}(X),-)\simeq D\oHom_\L(-,\tau_n(X)).\end{equation}
In particular, $\tau_n(X)\neq0$ implies $\Ext^n_\L(X,\tau_n(X))\simeq
D\overline{\End}_\L(\tau_n(X))\neq0$.
We have the following characterizations of $n$-fold almost split extensions.

\begin{lem}\label{characterize n-fold}
Let $X$ and $Y$ be modules in $\mod\L$, and $\eta$ a non-zero element in
$\Ext^n_\L(X,Y)$. Then the following conditions are equivalent.
\begin{enumerate}[\rm(i)]
\item $\eta$ is a right $n$-fold almost split extension of $X$.
\item $\Soc\Ext_\L^n(X,-)$ is a simple functor on $\mod\L$ and generated by $\eta$.
\item $\tau_n(X)$ is indecomposable, and $f\eta=0$ holds
in $\Ext^n_\L(X,Z)$ for every morphism $f\colon Y\to Z$ in the radical of $\mod\L$.
\end{enumerate}
If $Y$ is indecomposable, then the following condition is also equivalent.
\begin{enumerate}[\rm(iv)]
\item $Y\simeq\tau_n(X)$ and $\eta$ is a non-zero element in
$\Soc{}_{\End_\L(Y)}\Ext^n_\L(X,Y)$.
\end{enumerate}
\end{lem}

\begin{proof}
(i) is equivalent to (ii): 
Clearly (i) is equivalent to that $\eta$ is contained in every non-zero subfunctor of
$\Ext^n_\L(X,-)$. This is equivalent to (ii) since any subfunctor of $\Ext^n_\L(X,-)$
has a non-zero socle by \eqref{ExtC}.

(ii) is equivalent to (iii): Since \eqref{ExtC} holds,
$\Soc\Ext^n_\L(X,-)$ is simple if and only if $\tau_n(X)$ is indecomposable in $\omod\L$.
In this case, $\eta$ belongs to the socle if and only if it is annihilated by the radical of
$\mod\L$. Thus the assertion follows.

(ii) is equivalent to (iv): This is immediate from \eqref{ExtC}.
\end{proof}

Now we have the following existence and uniqueness result of right
$n$-fold almost split extensions.

\begin{prop}\label{right n-fold}
Let $\L$ be an Artin algebra, and $n\ge1$.
A module $X$ in $\mod\L$ has a right $n$-fold almost split extension if and only if
$\tau_n(X)$ is indecomposable.
If these conditions are satisfied, then the following assertions hold.
\begin{enumerate}[\rm(a)]
\item Let $Y$ be an indecomposable module in $\mod\L$ and $\eta$ an element in
$\Ext^n_\L(X,Y)$. Then $\eta$ is a right $n$-fold almost split extensions of $X$
if and only if $Y\simeq\tau_n(X)$ and $\eta$ is a non-zero element in
$\Soc_{\End_\L(Y)}\Ext^n_\L(X,Y)$.
\item For $i=1,2$, assume that $\eta_i$ in $\Ext^n_\L(X,Y_i)$ is a right
$n$-fold almost split extension of $X$ with an indecomposable module $Y_i$.
Then there exists an isomorphism $f\colon Y_1\to Y_2$ such that $\eta_2=f\eta_1$.
\end{enumerate}
\end{prop}

\begin{proof}
``Only if'' part follows from Lemma \ref{characterize n-fold}(i)$\Rightarrow$(iii), and
``if'' part follows from Lemma \ref{characterize n-fold}(iv)$\Rightarrow$(i).

(a) This was shown in Lemma \ref{characterize n-fold}(i)$\Leftrightarrow$(iv).

(b) By our assumption, there exists $f\colon Y_1\to Y_2$ and $g\colon Y_2\to Y_1$ 
satisfying $\eta_2=f\eta_1$ and $\eta_1=g\eta_2$.
Since $(1-gf)\eta_1=0$ holds, $1-gf$ is not an automorphism of $Y_1$, and hence
$gf$ is an automorphism of $Y_1$.
Similarly $fg$ is an automorphism of $Y_2$, and hence $f$ and $g$ are isomorphisms.
\end{proof}

We record dual results.

\begin{lem}
Let $X$ and $Y$ be modules in $\mod\L$, and $\eta$ a non-zero element
in $\Ext^n_\L(X,Y)$. Then the following conditions are equivalent.
\begin{enumerate}[\rm(i)]
\item $\eta$ is a left $n$-fold almost split extension of $Y$.
\item $\Soc\Ext_\L^n(-,Y)$ is a simple functor on $\mod\L$ and generated by $\eta$.
\item $\tau_n^-(Y)$ is indecomposable, and $\eta f=0$ holds 
in $\Ext^n_\L(Z,Y)$ for every morphism $f\colon Z\to X$ in the radical of $\mod\L$.
\end{enumerate}
If $X$ is indecomposable, then the following condition is also equivalent.
\begin{enumerate}[\rm(iv)]
\item $X\simeq\tau_n^-(Y)$ and $\eta$ is a non-zero element in
$\Soc\Ext^n_\L(X,Y)_{\End_\L(X)}$.
\end{enumerate}
\end{lem}

We also have the following dual result for left $n$-fold almost split extensions.

\begin{prop}\label{left n-fold}
Let $\L$ be an Artin algebra, and $n\ge1$.
A module $Y$ in $\mod\L$ has a left $n$-fold almost split extension if and only if
$\tau_n^-(Y)$ is indecomposable.
If these conditions are satisfied, then the following assertions hold.
\begin{enumerate}[\rm(a)]
\item Let $X$ be an indecomposable module in $\mod\L$ and $\eta$ an element in
$\Ext^n_\L(X,Y)$. Then $\eta$ is a left $n$-fold almost split extensions of $Y$
if and only if $X\simeq\tau_n^-(Y)$ and $\eta$ is a non-zero element in
$\Soc\Ext^n_\L(X,Y)_{\End_\L(X)}$.
\item For $i=1,2$, assume that $\eta_i$ in $\Ext^n_\L(X_i,Y)$ is a left
$n$-fold almost split extension of $Y$ with an indecomposable module $X_i$.
Then there exists an isomorphism $f\colon X_2\to X_1$ such that $\eta_2=\eta_1f$.
\end{enumerate}
\end{prop}

The following easy observation shows that all $n$-fold almost
split extensions can be obtained from those between indecomposable modules.
\begin{lem}\label{between indecomposable}
  Let $X=\bigoplus_{i=1}^mX_i$ and $Y=\bigoplus_{j=1}^\ell Y_j$ be
  modules in $\mod\L$ with indecomposable summands $X_i$ and $Y_j$,
  and $\eta=(\eta_{ij})_{i,j}$ a non-zero element in
  $\Ext^n_\L(X,Y)=\bigoplus_{i,j}\Ext^n_\L(X_i,Y_j)$.
\begin{enumerate}[\rm(a)]
\item $\eta$ is a right $n$-fold almost split extension of $X$ if and
  only if the following conditions hold.
\begin{enumerate}[\rm(i)]
\item There exists $1\le i_0\le m$ such that $\Ext^n_\L(X_i,-)=0$ for every $i\neq i_0$.
\item $\eta_{i_0j}$ is either zero or a right $n$-fold almost split extension of $X_{i_0}$
for every $j$.
\end{enumerate}
\item $\eta$ is a left $n$-fold almost split extension of $Y$ if and only if the following conditions hold.
\begin{enumerate}[\rm(i)]
\item There exists $1\le j_0\le\ell$ such that $\Ext^n_\L(-,Y_j)=0$ for every $j\neq j_0$.
\item $\eta_{ij_0}$ is either zero or a left $n$-fold almost split extension of $Y_{j_0}$
for every $j$.
\end{enumerate}
\item $\eta$ is an $n$-fold almost split extension if and only if the following conditions hold.
\begin{enumerate}[\rm(i)]
\item There exist $1\le i_0\le m$ and $1\le j_0\le\ell$ such that $\Ext^n_\L(X_i,-)=0$
for every $i\neq i_0$ and $\Ext^n_\L(-,Y_j)=0$ for every $j\neq j_0$.
\item $\eta_{i_0j_0}$ is an $n$-fold almost split extension.
\end{enumerate}
\end{enumerate}
\end{lem}

\begin{proof}
  (a) The ``if'' part follows directly from
  definition.  We show the "only if'' part. It follows from
  Lemma \ref{characterize n-fold}(iii) that $\tau_n(X)$ is
  indecomposable, and hence (i) holds. One can check (ii) easily from
  definition.

(b) This is dual of (a).

(c) This follows immediately from (a) and (b).
\end{proof}

Now we consider $n$-fold almost split extensions.

\begin{lem}\label{compare socle}
Let $X$ and $Y$ be indecomposable modules in $\mod\L$ satisfying
$Y\simeq\tau_n(X)$ and $X\simeq\tau_n^-(Y)$.
Then $\Soc{}_{\End_\L(Y)}\Ext^n_\L(X,Y)=\Soc\Ext^n_\L(X,Y)_{\End_\L(X)}$ holds.
\end{lem}

\begin{proof}
Let $E_X:=\End_\L(X)$ and $E_Y:=\End_\L(Y)$.
Since $\Soc{}_{E_Y}\Ext^n_\L(X,Y)\simeq D(\Top\oHom_\L(Y,\tau_n(X))_{E_Y})$ 
is simple, it is contained in any non-zero submodule of ${}_{E_Y}\Ext^n_\L(X,Y)$. 
Since $\Soc\Ext^n_\L(X,Y)_{E_X}$ is also a non-zero submodule of
${}_{E_Y}\Ext^n_\L(X,Y)$, we have
$\Soc{}_{E_Y}\Ext^n_\L(X,Y)\subset\Soc\Ext^n_\L(X,Y)_{E_X}$.
Since the reverse inclusion holds by the same argument, the desired equality holds.
\end{proof}

We have the following characterization of $n$-fold almost split extensions.

\begin{thm}\label{n-fold}
Let $\L$ be an Artin algebra, and $n\ge1$.
\begin{enumerate}[\rm(a)]
\item Let $X$ be an indecomposable module in $\mod\L$.
There exists an $n$-fold almost split extension in $\Ext^n_\L(X,Y)$ for
some module $Y$ in $\mod\L$ if and only if $\tau_n(X)$ is
indecomposable and $\tau_n^-\tau_n(X)\simeq X$.
If these conditions are satisfied, then any right $n$-fold almost split extension $\eta$
of $X$ in $\Ext^n_\L(X,Y)$ with an indecomposable module $Y$
is an $n$-fold almost splits extension.
\item Let $Y$ be an indecomposable module in $\mod\L$.
There exists an $n$-fold almost split extension in $\Ext^n_\L(X,Y)$ for
some module $X$ in $\mod\L$ if and only if $\tau_n^-(Y)$ is
indecomposable and $\tau_n\tau_n^-(Y)\simeq Y$.
If these conditions are satisfied, then any left $n$-fold almost split extension $\eta$
of $Y$ in $\Ext^n_\L(X,Y)$ with an indecomposable module $X$
is an $n$-fold almost splits extension.
\end{enumerate}
\end{thm}

\begin{proof}
(a) We show the ``if'' part. Let $Y:=\tau_n(X)$.
Then $\Soc{}_{\End_\L(Y)}\Ext^n_\L(X,Y)=\Soc\Ext^n_\L(X,Y)_{\End_\L(X)}$ holds by
Lemma \ref{compare socle}, and any non-zero element is a right $n$-fold almost
split extension of $X$ by Proposition \ref{right n-fold}(a) and left $n$-fold almost
split extension of $X$ by Proposition \ref{left n-fold}(a) .

This argument also gives a proof of the second statement.

We show the ``only if'' part.
By Lemma \ref{between indecomposable}, we can assume that $Y$ is indecomposable.
By Propositions \ref{right n-fold}(a) and \ref{left n-fold}(a), we have $Y\simeq\tau_n(X)$ 
and $X\simeq\tau_n^-(Y)$. Thus the assertions hold.

(b) This is dual of (a).
\end{proof}

Applying these general results to $n$-precluster tilting subcategories,
we have the following result.

\begin{thm}
Let $\L$ be an Artin algebra and $\C$ an $n$-precluster tilting subcategory in $\mod\L$
with $n\ge1$.
\begin{enumerate}[\rm(a)]
\item Every indecomposable object $X$ in $\Z(\C)\setminus\P(\L)$ has
an $n$-fold almost split extension in $\Ext^n_\L(X,\tau_n(X))$.
\item Every indecomposable object $Y$ in $\Z(\C)\setminus\I(\L)$ has
an $n$-fold almost split extension in $\Ext^n_\L(\tau_n^-(Y),Y)$.
\end{enumerate}
\end{thm}

\begin{proof}
Since $\tau_n(X)$ and $\tau_n^-(Y)$ are indecomposable by the equivalences
\eqref{taun equivalences}, the assertions follow immediate from
Propositions \ref{right n-fold} and \ref{left n-fold}.
\end{proof}

Now we consider exact sequences representing a right $n$-fold
almost split extension of a module $X$ in $\mod\L$ such that
$\tau_n(X)$ is indecomposable.  We consider a projective resolution
\[0\to\Omega_\L^{n-1}(X)\to P_{n-2}\to\cdots\to P_0\to X\to0.\]
Then $\Omega_\L^{n-1}(X)=Z\oplus P$ holds for an 
indecomposable module $Z$ in $(\mod\L)\setminus\P(\L)$ and a module $P$ in $\P(\L)$.
Let
\[0\to\tau_n(X)\to E\oplus P\to\Omega_\L^{n-1}(X)\to0\]
be a direct sum of an almost split sequence $0\to \tau_n(X)\to E\to Z\to0$ and
a split exact sequence $0\to 0\to P\to P\to0$. Taking the Yoneda product of these
exact sequences, we have an exact sequence
\begin{equation}\label{represent right n-fold}
0\to\tau_n(X)\to E\oplus P\to P_{n-2}\to\cdots\to P_0\to X\to0
\end{equation}
which represents a right $n$-fold almost split extensions of a module $X$.
In fact, one can apply Lemma \ref{characterize n-fold}(iii)$\Rightarrow$(i).

Dually, a left $n$-fold almost split extensions of a module $Y$ in $\mod\L$ such that
$\tau_n^-(Y)$ is indecomposable is represented by an exact sequence
\begin{equation}\label{represent left n-fold}
0\to Y\to I^0\to\cdots\to I^{n-2}\to E'\oplus I\to\tau_n^-(Y)\to0
\end{equation}
with modules $I^0,\ldots,I^{n-2}$ and $I$ in $\I(\L)$.

These constructions are quite general. However, 
even if $X$ (respectively $Y$) in \eqref{represent right n-fold} (respectively 
\eqref{represent left n-fold}) belongs to an $n$-precluster tilting subcategory $\C$
of $\mod\L$, the module $E$ (respectively $E'$) does not necessarily belong
to even $\Z(\C)$. 

Under a certain condition on $X$ (respectively $Y$), the next result gives much nicer 
representatives, which satisfies similar properties of $n$-almost split sequences in
$n$-cluster tilting subcategories. This is an analog of \cite[Theorem 3.3.1]{I}.

\begin{thm}\label{represent n-fold}
Let $\C$ be an $n$-precluster tilting subcategory of $\mod\L$,
$X$ an indecomposable module in $\Z(\C)\setminus \P(\L)$, and $Y:=\tau_n(X)$ the corresponding
indecomposable module in $\Z(\C)\setminus \I(\L)$.
\begin{enumerate}[\rm(a)]
\item For each $0\le i\le n-1$,
an $n$-fold almost split extension in $\Ext^n_\L(X,Y)$ can be represented as
\begin{equation}\label{CZC}
0\to Y\to C_{n-1}\to \cdots \to C_{i+1}\to Z_i\to C_{i-1}\to \cdots \to C_0\to X\to 0
\end{equation}
with $Z_i$ in $\Z(\C)$ and $C_j$ in $\C$ for each $j\neq i$.
\item The following sequences are exact.
\begin{align*}
&0\to \Hom_\L(\C,Y)\to \Hom_\L(\C,C_{n-1})\to\cdots\to \Hom_\L(\C,C_{i+1})\to\Hom_\L(\C,Z_i)\\
&\qquad\qquad\qquad\to\Hom_\L(\C,C_{i-1})\to\cdots\to\Hom_\L(\C,C_0)\to\rad_\L(\C,X)\to 0,\\
&0\to \Hom_\L(X,\C)\to \Hom_\L(C_0,\C)\to\cdots\to\Hom_\L(C_{i-1},\C)\to\Hom_\L(Z_i,\C)\\
&\qquad\qquad\qquad\to\Hom_\L(C_{i+1},\C)\to\cdots\to\Hom_\L(C_{n-1},\C)\to\rad_\L(Y,\C)\to 0.
\end{align*}
\item If $X$ and $Y$ do not belong to $\C$, then the $n$-fold almost split extension in \emph{(a)}
can be given as a Yoneda product of a minimal projective resolution of $X$ in $\Z(\C)$
\begin{equation}\label{right part}
0\to\Omega_{\Z(\C)}^{i}(X)\to C_{i-1}\to \cdots \to C_{0}\to X\to 0,
\end{equation}
an almost split sequence in $\Z(\C)$
\begin{equation}\label{middle part}
0\to\Omega_{\Z(\C)}^{-(n-i-1)}(Y)\to Z_i\to\Omega_{\Z(\C)}^{i}(X)\to 0,
\end{equation}
and a minimal injective coresolution of $Y$ in $\Z(\C)$
\begin{equation}\label{left part}
0\to Y\to C_{n-1}\to \cdots \to C_{i+1}\to\Omega_{\Z(\C)}^{-(n-i-1)}(Y)\to 0.
\end{equation}
\end{enumerate}
\end{thm}

\begin{proof}
(a) By \eqref{represent right n-fold}, an $n$-fold almost split extension is given by the Yoneda product of
$0\to Y\to E\to\Omega^{n-1}_\L(X)\to0$ (where $E\oplus P$ is written as $E$) and
$0\to\Omega^{n-1}_\L(X)\to P_{n-2}\to\cdots\to P_0\to X\to0$.
Applying Corollary \ref{approximation sequence} to $E$, we have an exact sequence
\[0\to E\xrightarrow{f^0}C^0\xrightarrow{f^1}\cdots\xrightarrow{f^{i-1}}C^{i-1}\xrightarrow{f^i}
Z^{i}\xrightarrow{f^{i+1}}C^{i+1}\xrightarrow{f^{i+2}}\cdots\xrightarrow{f^{n-1}}C^{n-1}\to0\]
where $Z^i$ is in $\Z(\C)$, $C^j$ is in $\C$ and $E^j:=\Im f^j$ is in ${}^{\perp_j}\C$. We write $C^i:=Z^i$ for simplicity.
We have the following pushout diagram.
\[\xymatrix@R1em{
&0\ar[d]&0\ar[d]&0\ar[d]\\
0\ar[r]& Y\ar[r]\ar@{=}[d]&E\ar[r]\ar[d]&\Omega_\L^{n-1}(X)\ar[r]\ar[d]&0\\
0\ar[r]& Y\ar[r]&C^0\ar[r]\ar[d]&W\ar[r]\ar[d]&0\\
&& E^1\ar@{=}[r]\ar[d]&E^1\ar[d]\\
&&0&0
}\]
Since $\Ext^1_\L(E^j,P_{n-j})=0$ holds for any $j\ge2$, a Horseshoe Lemma-type argument gives
the following commutative diagram of exact sequences.
\[\xymatrix@R1em@C1.2em{
&0\ar[d]&0\ar[d]&&0\ar[d]&0\ar[d]\\
0\ar[r]&\Omega_\L^{n-1}(X)\ar[r]\ar[d]&P_{n-2}\ar[r]\ar[d]&\cdots\ar[r]&P_0\ar[r]\ar[d]&X\ar[r]\ar@{=}[d]&0\\
0\ar[r]&W\ar[r]\ar[d]&C^1\oplus P_{n-2}\ar[r]\ar[d]&\cdots\ar[r]&C^{n-1}\oplus P_0\ar[r]\ar[d]&X\ar[r]\ar[d]&0\\
0\ar[r]&E^1\ar[r]\ar[d]&C^1\ar[r]\ar[d]&\cdots\ar[r]&C^{n-1}\ar[r]\ar[d]&0\\
&0&0&&0
}\]
The Yoneda product
$0\to Y\to C^0\to C^1\oplus P_{n-2}\to\cdots\to C^{n-1}\oplus P_0\to
X\to0$ of the middle sequences of the diagrams above represents the
same class in $\Ext^n_\L(X,Y)$ as the sequence \eqref{represent right
  n-fold}.  Thus it represents an $n$-fold almost split extension.

(b) This is easily checked (see e.g. \cite[Lemma 3.2]{I}).

(c) It suffices to show that the Yoneda product of \eqref{right part}, \eqref{middle part}
and \eqref{left part} is an $n$-fold almost split extension.
We have functorial isomorphisms
\begin{eqnarray*}
\Ext^n_\L(X,\tau_n(X))&=&\Hom_{\U(\C)}(X,\tau_n(X)[n])
\simeq \Hom_{\U(\C)}(X[-i],\tau_n(X)[n-i])\\
&=&\Hom_{\U(\C)}(\Omega_{\Z(\C)}^{i}(X),\Omega_{\Z(\C)}^{-(n-i-1)}\tau_n(X)[1])
\end{eqnarray*}
which induces an isomorphism
\begin{eqnarray*}
&&\Soc{}_{\End_\L(\tau_n(X))}\Ext^n_\L(X,\tau_n(X))\\
&\simeq&\Soc{}_{\End_{\U(\C)}(\Omega_{\Z(\C)}^{-(n-i-1)}\tau_n(X))}\Hom_{\U(\C)}(\Omega_{\Z(\C)}^{i}(X),\Omega_{\Z(\C)}^{-(n-i-1)}\tau_n(X)[1]).
\end{eqnarray*}
Since the almost split sequence belongs to the right hand side, our Yoneda product
belongs to the left hand side.
Thus it is an $n$-fold almost split extension of $X$ by Theorem \ref{n-fold}.
\end{proof}

Note that the sequences Theorem \ref{represent n-fold}(b) are not exact in general
if we replace $\C$ by $\Z(\C)$. Thus the sequences \eqref{CZC} is not necessarily
$n$-exact in $\Z(\C)$ in the sense of Jasso \cite{J2}.

In the rest of this section we let $\C = \add M$ be a finite
$n$-precluster tilting subcategory of $\mod\L$, and let
$\G=\End_\L(M)$ be the corresponding $n$-minimal Auslander--Gorenstein algebra.
We show that the category $\CM\G$ of maximal Cohen--Macaulay modules
over $\G$ has $n$-fold almost split extensions.  This is used in the
next section to classify the $n$-minimal Auslander--Gorenstein algebras.  We
relate the $n$-fold almost split extensions to the corresponding
$n$-precluster tilting subcategory and transfer properties and
constructions, as that of the $n$-fold Auslander--Reiten translate.

The dualities $\Hom_{\L}(-,M)\colon\Z(\C)\to\CM\G$ and
$\Hom_{\G}(-,M)\colon\CM\G\to\Z(\C)$ induce dualities
\[\underline{\Z(\C)}\longleftrightarrow(\CM\G)/[\add I]\ \mbox{ and }\ 
\overline{\Z(\C)}\longleftrightarrow(\CM\G)/[\add Q],\]
where $I:={}_\G M$ is an additive generator
of $\P(\G)\cap\I(\G)$ and $Q:=\Hom_\L(D\L,M)\simeq\Hom_{\G^{\op}}(DI,\G)\in\P(\G)$.

\begin{defin}\label{define sigma}
Using the equivalences $\tau_n\colon \underline{\Z(\C)}\to \overline{\Z(\C)}$ and
  $\tau^-_n\colon \overline{\Z(\C)}\to \underline{\Z(\C)}$ 
given in Theorem \ref{Z(C)}, we define equivalences
  \[\sigma_n^-\colon (\CM\G)/[\add I]\to (\CM\G)/[\add Q]\ \mbox{ and }\ \sigma_n\colon (\CM\G)/[\add Q] \to  (\CM\G)/[\add I]\]
  making the following diagrams commutative up to isomorphisms of
  functors.
\begin{equation}\label{sigma}
\xymatrix@R=1.5em@C=60pt{
    \underline{\Z(\C)} \ar[r]^{\Hom_\L(-,M)} \ar@<-0.5ex>[d]_{\tau_n} & 
    (\CM\G)/[\add I]\ar@<-0.5ex>[d]^{\sigma^-_n}\\
    \overline{\Z(\C)} \ar[r]^-{\Hom_\L(-,M)}  &
    (\CM\G)/[\add Q]
}\ \ \ \xymatrix@R=1.5em@C=60pt{
    \underline{\Z(\C)} \ar[r]^{\Hom_\L(-,M)}  & (\CM\G)/[\add I]\\
    \overline{\Z(\C)} \ar[r]^-{\Hom_\L(-,M)} \ar@<-0.5ex>[u]^{\tau^-_n} &
    (\CM\G)/[\add Q].\ar@<-0.5ex>[u]_{\sigma_n} 
}\end{equation}
\end{defin}

Let $\G$ be an $n$-minimal Auslander--Gorenstein Artin algebra for some
$n\geq 1$.  Then we use Theorem \ref{represent n-fold} to
construct $n$-fold almost split extensions in $\CM\G$.

\begin{cor}\label{prop:nfoldalmostext-in-gammaperp} 
Let $\G$ be an $n$-minimal Auslander--Gorenstein Artin algebra with $n\geq1$,
$X$ an indecomposable module in $(\CM\G)\setminus(\add Q)$
and $Y:=\sigma_n(X)$ an indecomposable module in $(\CM\G)\setminus(\add I)$.
\begin{enumerate}[\rm(a)]
\item If $X$ does not belong to $\P(\G)$, then there exists an
  $n$-fold almost split extension in $\Ext^n_\G(X,Y)$. 
For each $0\le i\le n-1$, it is represented as
\[0\to Y\to P^0\to \cdots \to P^{i-1}\to E^i\to P^{i+1}\to \cdots \to P^{n-1}\to  X\to 0\]
with $P_j$ for all $j$ in $\P(\G)$ and $E_i$ in $\CM\G$. 
\item If $X$ belongs to $\P(\G)$, then there exists an exact sequence
\[0\to Y\to  P^0\to \cdots \to P^{i-1}\to E^i\to P^{i+1}\to \cdots \to P^{n-1}\to X\to X/\rad X\to 0,\]
with $P_j$ for all $j$ in $\P(\G)$ and $E_i$ in $\CM\G$.
\end{enumerate}
\end{cor}

We call the sequence in (b) a \emph{fundamental sequence}.
\begin{proof}
  Both assertions follow from Theorem \ref{represent n-fold} and the
  duality between $\CM\G$ and $\mathcal{Z}(\C)$ for the corresponding
  $n$-precluster tilting subcategory $\C$ (Theorem \ref{CM duality}).
\end{proof}

\section{Four classes of $n$-minimal Auslander--Gorenstein algebras}\label{section:6}

$1$-Auslander--Gorenstein Artin algebras were classified in \cite{ASIV}
in the class of selfinjective Artin algebras and three different
disjoint classes using the correspondence with the
$\tau$-selfinjective Artin algebras.  This section is devoted to
extending this characterization to $n$-minimal Auslander--Gorenstein algebras
for $n\geq 1$.

Suppose $\G$ is an $n$-minimal Auslander--Gorenstein algebra.
By Theorem \ref{AS correspondence}, there exist
an Artin algebra $\L$ and an $n$-precluster tilting $\L$-module
$M$ with $\G\simeq \End_\L(M)$. By Proposition
\ref{lem:preclustercategory}(d), $\add M = \add\{\P_n\vee\I_n,N\}$ holds
for a $\L$-module $N$ satisfying $\tau_n(N)\simeq N$.  Then four
distinct cases can occur:
\begin{enumerate}[\rm(A)]
\item $\L$ is selfinjective and $N=0$.
\item $\L$ is selfinjective and $N\neq 0$.
\item $\L$ is non-selfinjective and $N=0$. 
\item $\L$ is non-selfinjective and $N\neq 0$. 
\end{enumerate}
This also gives rise to a coarser division, namely (A) and (C)
together, or in other words when $N = 0$.

We characterize all these four cases in terms of properties of $\G$,
and in particular we show that an algebra of type (B) or (D) is
constructed from an algebra of type (A) or (C), respectively.

The case (A) is easy, as this occurs if and only if $\G$ is a
selfinjective algebra by Proposition \ref{prop:endoring}.
So we move on to describe the correspondence given by type (B).
\begin{prop}\label{prop:case-B}
  Let $n\geq 1$. The bijection in \textup{Theorem \ref{AS
        correspondence}} restricts to a bijection between
\begin{eqnarray*}
&\{\textrm{finite $n$-precluster tilting subcategories of type \emph{(B)}}\}\ 
\mbox{ and }&\\
&\left\{\textrm{\parbox{0.8\textwidth}{
      non-selfinjective $n$-minimal Auslander--Gorenstein algebras $\G$ with
      an additive generator $I$ of $\P(\G)\cap\I(\G)$ satisfying
      $\add\Top {_\G I} = \add\Soc {_\G I}$}}\right\}.&
\end{eqnarray*}
\end{prop}

\begin{proof}
  Let $\L$ be a selfinjective Artin algebra with 
  non-projective module $N$ such that $M=\L\oplus N$ is an
  $n$-precluster tilting $\L$-module. We
  want to show that $\G=\End_\L(M)$ is a non-selfinjective
  $n$-minimal Auslander--Gorenstein algebra
  and $\add \Top {_\G I} = \add\Soc {_\G I}$,
  where $I$ is an additive generator of $\P(\G)\cap\I(\G)$. 

 Let $G=F_{ M}\subseteq \Ext^1_\L(-,-)$.
 By Proposition \ref{prop:endoring} we have that $\G$ is a
  non-selfinjective $n$-minimal Auslander--Gorenstein algebra,
  $M$ is a $G$-cotilting $\L$-module, and $I ={_\G M}$
  is an additive generator of $\P(\G)\cap\I(\G)$.
By \cite[Lemma 2.3(b)]{ASIV} we have that
\[\Hom_\G(-,I)\colon \Hom_\G(A,C)\to \Hom_\L(\Hom_\G(C,I),\Hom_\G(A,I))\]
is an isomorphism for all $C$ in $\mod\G$ and all $A$ in $\add I$. 
In particular for $A=I$ and noting that $\Hom_\G(I,I)\simeq \L$ with
$\L$ selfinjective, we infer that $\Hom_\G(I,C)=0$ if and only if
$\Hom_\G(C,I)=0$. It follows from this that $\add \Top {_\G I}=\add
\Soc{_\G I}$. 

Conversely, let $\G$ be a non-selfinjective $n$-minimal Auslander--Gorenstein algebra,
such that $\add\Top {_\G I} = \add\Soc {_\G I}$ where $I$ is an additive generator
of $\P(\G)\cap\I(\G)$. We want to show that $\L=\End_\G(I)$
is selfinjective, and $M={_\L I}\simeq \L\oplus N$ with
$N$ non-projective such that $M$ is an
$n$-precluster tilting $\L$-module.

Let $G=F_M\subseteq \Ext^1_\L(-,-)$.
By Proposition \ref{prop:redtoendoring} the module $M$ is an $n$-precluster tilting
$\L$-module and $M={_\L I}$ is a $G$-cotilting $\L$-module. The functor
$\Hom_\G(-,I)\colon \mod\G\to \mod\L$ induces a duality between
$\P(\G)$ and $\add M$. Since $M$ is a cogenerator, there exists $P$ in
$\P(\G)$ such that $\Hom_\G(P,I)\simeq D(\L_\L)$. 
By \cite[Lemma 2.3(b)]{ASIV} we have that
\[\Hom_\G(-,I)\colon \Hom_\G(P,C)\to \Hom_\L(\Hom_\G(C,I),\Hom_\G(P,I))\]
is an isomorphism for all $C$ in $\mod\G$. This implies by the choice
of $P$ that $\Hom_\G(P,C)=0$ if and only if
$\Hom_\G(C,I)=0$. Hence $\add \Top {_\G P} = \add \Soc {_\G I} = \add
\Top {_\G I}$, and we obtain that $\add P=\add I$. Therefore
$\I(\L)=\add\Hom_\G(P,I)=\P(\L)$ and $\L$ is selfinjective.
\end{proof}

\begin{remark}\label{from A to B}
  By Proposition \ref{prop:case-B}, if $\G$
  is an $n$-minimal Auslander--Gorenstein algebra of type (B), then
  there exists a selfinjective algebra $\L$ and a non-projective $n$-precluster tilting
  $\L$-module $\L\oplus N$ such that $\tau_n(N)\simeq N$ and
  $\G\simeq \End_\L(\L\oplus N)$.
\end{remark}

Now we complete the classification. 
\begin{thm}
  Let $\G$ be an $n$-minimal Auslander--Gorenstein algebra for
  $n\geq 1$. Denote by $I$ an additive generator of $\P(\G)\cap\I(\G)$,
  and let $\L=\End_\G(I)$.  Then four disjoint cases occur, where $\sigma_n$ is
  the functor given in \textup{Definition} \ref{define sigma}:
\begin{enumerate}[\rm(A)]
\item $\G$ is selfinjective, 
\item $\G$ is not selfinjective and $\add \Top {_\G I} = \add\Soc {_\G I}$, 
\item $\add \Top {_\G I} \neq \add\Soc {_\G I}$, and there exists no
  indecomposable projective $\G$-module $P$ such that
  $(\sigma_n)^t(P)\simeq P$ for some $t>0$, 
\item $\add \Top {_\G I} \not\neq\add\Soc {_\G I}$, and there exists an 
  indecomposable projective $\G$-module $P$ such that
  $(\sigma_n)^t(P)\simeq P$ for some $t>0$.
\end{enumerate}
In the cases \emph{(B)}, \emph{(C)} and \emph{(D)} above, the
correspondence given in \textup{Theorem \ref{AS correspondence}} shows that
the algebra $\G$ arises as follows.
\begin{enumerate}[\rm(a)]
\item In case \emph{(B)} there exists a selfinjective algebra $\L$
  with an $n$-precluster tilting $\L$-module $M$ such that
  $\G\simeq \End_\L(M)$.
\item In case \emph{(C)} there exists an algebra $\L$ with an $n$-precluster
  tilting $\L$-module $M$ with no non-zero $\tau_n$-periodic direct summand
  such that $\G \simeq \End_\L(M)$. 
\item In case \emph{(D)} there exists an algebra $\L$ with an $n$-precluster
  tilting $\L$-module $M$ with a non-zero $\tau_n$-periodic direct summand
  such that $\G \simeq \End_\L(M)$.
\end{enumerate}
\end{thm}
\begin{proof}
  The cases (A) and (B) are already discussed before in
  Proposition \ref{prop:case-B}.  The cases (A) and (B) occur if and
  only if $\add \Top {_\G I} = \add\Soc {_\G I}$. Hence the cases (C) and (D)
  occur if and only if $\add \Top {_\G I} \neq \add\Soc {_\G I}$.  Using this and the commutative diagram \eqref{sigma}, the characterizations of the cases (C) and (D) follow.
\end{proof}

\begin{remark}
  As we saw in Remark \ref{from A to B}, any algebra of type (B) can be
  constructed from an algebra of type (A). 
  Now we argue that type (D) and type (C) are related in a
  similar fashion.  Let $\G$ be an $n$-minimal Auslander--Gorenstein algebra of type
  (D).  Let $I$ denote an additive generator of $\P(\G)\cap\I(\G)$, and let
  $\L=\End_\G(I)$ and $M={_\L I}$.  Decompose $M = M_0\oplus
  N$ with $\add M_0 = \P_n\vee\I_n$ and $\tau_n(N) \simeq N$.  Furthermore,
  $\G_0=\End_\G(M_0)$ is an $n$-minimal Auslander--Gorenstein algebra of type (C)
  since $M_0$ is also an $n$-precluster tilting $\L$-module.
  Now we claim that $\G$ can be constructed from $\G_0$.

Since $M_0$ is a cogenerator for $\mod\L$, the functor 
\[\Hom_\L(-,M_0)\colon \mod\L \to \mod\G_0\]
is full and faithful.  This implies that 
\[\G =\Hom_\L(M,M) \simeq \End_{\G_0}(\Hom_\L(M,M_0)).\]
Hence $\G$ is constructed from the algebra $\G_0$ which is of type
(C). 
\end{remark}

\end{document}